\documentclass[a4paper]{amsart}
\usepackage{amscd}
\usepackage{amsmath}
\usepackage{amssymb}
\usepackage{amsthm}
\usepackage{bbm}
\usepackage{stmaryrd}

\usepackage[T1]{fontenc}

\newif\ifpdf
\ifx\pdfoutput\undefined
   \pdffalse        
\else
   \pdfoutput=1     
   \pdftrue
\fi

\ifpdf
   \usepackage[pdftex]{graphicx}
\else
   \usepackage{graphicx}
\fi

\numberwithin{equation}{section} \swapnumbers

\newtheorem{satz}{Satz}[section]

\newtheorem{theorem}[satz]{Theorem}
\newtheorem{proposition}[satz]{Proposition}

\newtheorem{lemma}[satz]{Lemma}
\newtheorem{assumption}[satz]{Assumption}

\newtheorem{definition}[satz]{Definition}

\newtheorem{remark}[satz]{Remark}

\newcommand{\bbr}{\mathbb{R}}
\newcommand{\bbe}{\mathbb{E}}
\newcommand{\bbn}{\mathbb{N}}
\newcommand{\bbp}{\mathbb{P}}

\newcommand{\cala}{\mathcal{A}}

\newcommand{\cald}{\mathcal{D}}
\newcommand{\cale}{\mathcal{E}}
\newcommand{\calf}{\mathcal{F}}
\newcommand{\calh}{\mathcal{H}}

\newcommand{\bc}{{\rm bc}(\{ e_l \}_{l \in \bbn})}
\newcommand{\ubc}{{\rm ubc}(\{ e_l \}_{l \in \bbn})}

\newcommand{\Lip}{{\rm Lip}}
\newcommand{\LG}{{\rm LG}}
\newcommand{\loc}{{\rm loc}}
\newcommand{\B}{{\rm B}}
\newcommand{\F}{{\rm F}}
\newcommand{\G}{{\rm G}}
\newcommand{\supp}{{\rm supp}}
\newcommand{\Id}{{\rm Id}}

\newcommand{\bbI}{\mathbbm{1}}

\begin{document}

\hyphenation{rea-li-za-tion pa-ra-me-tri-za-ti-ons pa-ra-me-tri-za-ti-on Schau-der}

\title[Invariance of closed convex cones]{Invariance of closed convex cones for stochastic partial differential equations}
\author{Stefan Tappe}
\address{Leibniz Universit\"{a}t Hannover, Institut f\"{u}r Mathematische Stochastik, Welfengarten 1, 30167 Hannover, Germany}
\email{tappe@stochastik.uni-hannover.de}
\thanks{I am grateful to an anonymous referee for the careful study of my paper and the valuable comments and suggestions.}
\begin{abstract}
The goal of this paper is to clarify when a closed convex cone is invariant for a stochastic partial differential equation (SPDE) driven by a Wiener process and a Poisson random measure, and to provide conditions on the parameters of the SPDE, which are necessary and sufficient.
\end{abstract}

\maketitle\thispagestyle{empty}

\section{Introduction}\label{sec-intro}

Consider a semilinear stochastic partial differential equation (SPDE) of the form
\begin{align}\label{SPDE}
\left\{
\begin{array}{rcl}
dr_t & = & ( A r_t + \alpha(r_t) ) dt + \sigma(r_t) dW_t + \int_E \gamma(r_{t-},x) (\mu(dt,dx) - F(dx)dt) \medskip
\\ r_0 & = & h_0
\end{array}
\right.
\end{align}
driven by a trace class Wiener process $W$ and a Poisson random measure $\mu$ on some mark space $E$ with compensator $dt \otimes F(dx)$. The state space of the SPDE (\ref{SPDE}) is a separable Hilbert space $H$, and the operator $A$ is the generator of a strongly continuous semigroup $(S_t)_{t \geq 0}$ on $H$. We refer to Section \ref{sec-ass} for more details concerning the mathematical framework.

In applications, one is often interested in the question when a certain subset of the state space is invariant for the SPDE (\ref{SPDE}), and frequently it turns out that this subset is a closed convex cone. For example, when modeling the evolution of interest rate curves, a desirable feature is that the model produces nonnegative interest curves; or when modeling multiple yield curves, it is desirable to have spreads which are ordered with respect to different tenors.

In order to translate these ideas into mathematical terms, let $K \subset H$ be a closed convex cone of the state space $H$. We say that the cone $K$ is invariant for the SPDE (\ref{SPDE}) if for each starting point $h_0 \in K$ the solution process $r$ to (\ref{SPDE}) stays in $K$. The goal of this paper is to clarify when the cone $K$ is invariant for the SPDE (\ref{SPDE}), and to provide conditions on the parameters $(A,\alpha,\sigma,\gamma)$ -- or, equivalently, on $((S_t)_{t \geq 0},\alpha,\sigma,\gamma)$ -- of the SPDE (\ref{SPDE}), which are necessary and sufficient.

Stochastic invariance of a given subset $K \subset H$ for jump-diffusion SPDEs (\ref{SPDE}) has already been studied in the literature, mostly for diffusion SPDEs
\begin{align}\label{SPDE-Wiener}
\left\{
\begin{array}{rcl}
dr_t & = & ( A r_t + \alpha(r_t) ) dt + \sigma(r_t) dW_t \medskip
\\ r_0 & = & h_0
\end{array}
\right.
\end{align}
without jumps. The classes of subsets $K \subset H$, for which stochastic invariance has been investigated, can roughly be divided as follows:
\begin{itemize}
\item For a finite dimensional submanifold $K \subset H$ the stochastic invariance has been studied in \cite{Filipovic-inv} and \cite{Nakayama} for diffusion SPDEs (\ref{SPDE-Wiener}), and in \cite{FTT-manifolds} for jump-diffusion SPDEs (\ref{SPDE}). Here a related problem is the existence of a finite dimensional realization (FDR), which means that for each starting point $h_0 \in H$ a finite dimensional invariant manifold $K \subset H$ with $h_0 \in K$ exists. This problem has mostly been studied for the so-called Heath-Jarrow-Morton-Musiela (HJMM) equation from mathematical finance, and we refer, for example, to \cite{Bj_Sv, Bj_La, Filipovic, Filipovic-Teichmann-royal, Tappe-Wiener, Tappe-affine} for the existence of FDRs for diffusion SPDEs (\ref{SPDE-Wiener}), and, for example, to \cite{Tappe-Levy, Platen-Tappe, Tappe-affin-real} for the existence of FDRs for SPDEs driven by L\'{e}vy processes, which are particular cases of jump-diffusion SPDEs (\ref{SPDE}).

\item For an arbitrary closed subset $K \subset H$ the stochastic invariance has been studied for PDEs in \cite{Jachimiak-note}, and for diffusion SPDEs (\ref{SPDE-Wiener}) in \cite{Jachimiak} and -- based on the support theorem presented in \cite{Nakayama-Support} -- in \cite{Nakayama}. Both authors obtain the so-called stochastic semigroup Nagumo's condition (SSNC) as a criterion for stochastic invariance, which is necessary and sufficient. An indispensable assumption for the formulation of the SSNC is that the volatility $\sigma$ is sufficiently smooth; it must be two times continuously differentiable.

\item For a closed convex cone $K \subset H$ -- as in our paper -- the stochastic invariance has been studied in two particular situations on function spaces. In \cite{Milian} the state space $H$ is an $L^2$-space, $K$ is the closed convex cone of nonnegative functions, and its stochastic invariance is investigated for diffusion SPDEs (\ref{SPDE-Wiener}). In \cite{Positivity} the state space $H$ is a Hilbert space consisting of continuous functions, $K$ is also the closed convex cone of nonnegative functions, and its stochastic invariance is investigated for jump-diffusion SPDEs (\ref{SPDE}); a particular application in \cite{Positivity} is the positivity preserving property of interest rate curves from the aforementioned HJMM equation, which appears in mathematical finance.
\end{itemize}
In this paper, we provide a general investigation of the stochastic invariance problem for an arbitrary closed convex cone $K \subset H$, contained in an arbitrary separable Hilbert space $H$, for jump-diffusion SPDEs (\ref{SPDE}). Taking advantage of the structural properties of closed convex cones, we do not need smoothness of the volatility $\sigma$, as it is required in \cite{Jachimiak} and \cite{Nakayama}, and also in \cite{Positivity}.

In order to present our main result of this paper, let $K \subset H$ be a closed convex cone, and let $K^* \subset H$ be its dual cone
\begin{align}\label{dual-cone}
K^* = \bigcap_{h \in K} \{ h^* \in H : \langle h^*,h \rangle \geq 0 \}.
\end{align}
Then the cone $K$ has the representation
\begin{align}\label{cone-repr}
K = \bigcap_{h^* \in K^*} \{ h \in H : \langle h^*,h \rangle \geq 0 \}.
\end{align}
We fix a generating system $G^*$ of the cone $K$; that is, a subset $G^* \subset K^*$ such that the cone admits the representation
\begin{align}\label{cone-G}
K = \bigcap_{h^* \in G^*} \{ h \in H : \langle h^*,h \rangle \geq 0 \}.
\end{align}
In particular, we could simply take $G^* = K^*$. However, for applications we will choose a generating system $G^*$ which is as convenient as possible. Throughout this paper, we make the following assumptions:
\begin{itemize}
\item The semigroup $(S_t)_{t \geq 0}$ is pseudo-contractive; see Assumption \ref{ass-pseudo-contractive}.

\item The coefficients $(\alpha,\sigma,\gamma)$ are locally Lipschitz and satisfy the linear growth condition, which ensures existence and uniqueness of mild solutions to the SPDE (\ref{SPDE}); see Assumption \ref{ass-loc-Lip-LG}.

\item The cone $K$ is invariant for the semigroup $(S_t)_{t \geq 0}$; see Assumption \ref{ass-cone-semigroup}.

\item The cone $K$ is generated by an unconditional Schauder basis; see Assumption \ref{ass-Schauder-basis}.
\end{itemize}
We refer to Section \ref{sec-ass} for the precise mathematical framework. We define the set $D \subset G^* \times K$ as
\begin{align}\label{def-D}
D := \bigg\{ (h^*,h) \in G^* \times K : \liminf_{t \downarrow 0} \frac{\langle h^*,S_t h \rangle}{t} < \infty \bigg\}.
\end{align}
Since the cone $K$ is invariant for the semigroup $(S_t)_{t \geq 0}$, for all $(h^*,h) \in G^* \times K$ the limes inferior in (\ref{def-D}) exists with value in $\overline{\bbr}_+ = [0,\infty]$. Now, our main result reads as follows.

\begin{theorem}\label{thm-main}
Suppose that Assumptions \ref{ass-pseudo-contractive}, \ref{ass-loc-Lip-LG}, \ref{ass-cone-semigroup} and \ref{ass-Schauder-basis} are fulfilled. Then the following statements are equivalent:
\begin{enumerate}
\item[(i)] The closed convex cone $K$ is invariant for the SPDE (\ref{SPDE}).

\item[(ii)] We have
\begin{align}\label{main-1}
h + \gamma(h,x) \in K \quad \text{for $F$-almost all $x \in E$,} \quad \text{for all $h \in K$,}
\end{align}
and for all $(h^*,h) \in D$ we have
\begin{align}\label{main-3}
&\liminf_{t \downarrow 0} \frac{\langle h^*,S_t h \rangle}{t} + \langle h^*,\alpha(h) \rangle - \int_E \langle h^*,\gamma(h,x) \rangle F(dx) \geq 0,
\\ \label{main-4} &\langle h^*,\sigma^j(h) \rangle = 0, \quad j \in \bbn.
\end{align}
\end{enumerate}
\end{theorem}

Conditions (\ref{main-1})--(\ref{main-4}) are geometric conditions on the coefficients of the SPDE (\ref{SPDE}); condition (\ref{main-1}) concerns the behaviour of the solution process in the cone, and conditions (\ref{main-3}) and (\ref{main-4}) concern the behaviour of the solution process at boundary points of the cone:
\begin{itemize}
\item Condition (\ref{main-1}) is a condition on the jumps; it means that the cone $K$ is invariant for the functions $h \mapsto h + \gamma(h,x)$ for $F$-almost all $x \in E$. 

\item Condition (\ref{main-3}) means that the drift is inward pointing at boundary points of the cone.

\item Condition (\ref{main-4}) means that the volatilities are parallel at boundary points of the cone.
\end{itemize}
Figure \ref{fig-geometric} illustrates conditions (\ref{main-1})--(\ref{main-4}). Let us provide further explanations regarding the drift condition (\ref{main-3}). For this purpose, we fix an arbitrary pair $(h^*,h) \in D$. By the definition (\ref{def-D}) of the set $D$, we have $\langle h^*,h \rangle = 0$, indicating that we are at the boundary of the cone. 
\begin{itemize}
\item The drift condition (\ref{main-3}) implies
\begin{align}\label{main-2}
\int_E \langle h^*,\gamma(h,x) \rangle F(dx) < \infty.
\end{align}
This means that the jumps of the solution process at boundary points of the cone are of finite variation, unless they are parallel to the boundary.

\item If $h \in \cald(A)$, then the drift condition (\ref{main-3}) is fulfilled if and only if
\begin{align}\label{FV-cond-3}
\langle h^*, Ah + \alpha(h) \rangle - \int_E \langle h^*,\gamma(h,x) \rangle F(dx) \geq 0.
\end{align}
In view of condition (\ref{FV-cond-3}), we point out that $K \cap \cald(A)$ is dense in $K$.

\item If $h^* \in \cald(A^*)$, then the drift condition (\ref{main-3}) is fulfilled if and only if
\begin{align}\label{FV-cond-4}
\langle A^* h^*,h \rangle + \langle h^*, \alpha(h) \rangle - \int_E \langle h^*,\gamma(h,x) \rangle F(dx) \geq 0.
\end{align}
In particular, if $A^*$ is a local operator, then the drift condition (\ref{main-3}) is equivalent to
\begin{align}\label{FV-cond-5}
\langle h^*, \alpha(h) \rangle - \int_E \langle h^*,\gamma(h,x) \rangle F(dx) \geq 0.
\end{align} 
In any case, condition (\ref{FV-cond-5}) implies the drift condition (\ref{main-3}).
\end{itemize}
\begin{figure}[!ht]
 \centering
 \includegraphics[width=0.5\textwidth]{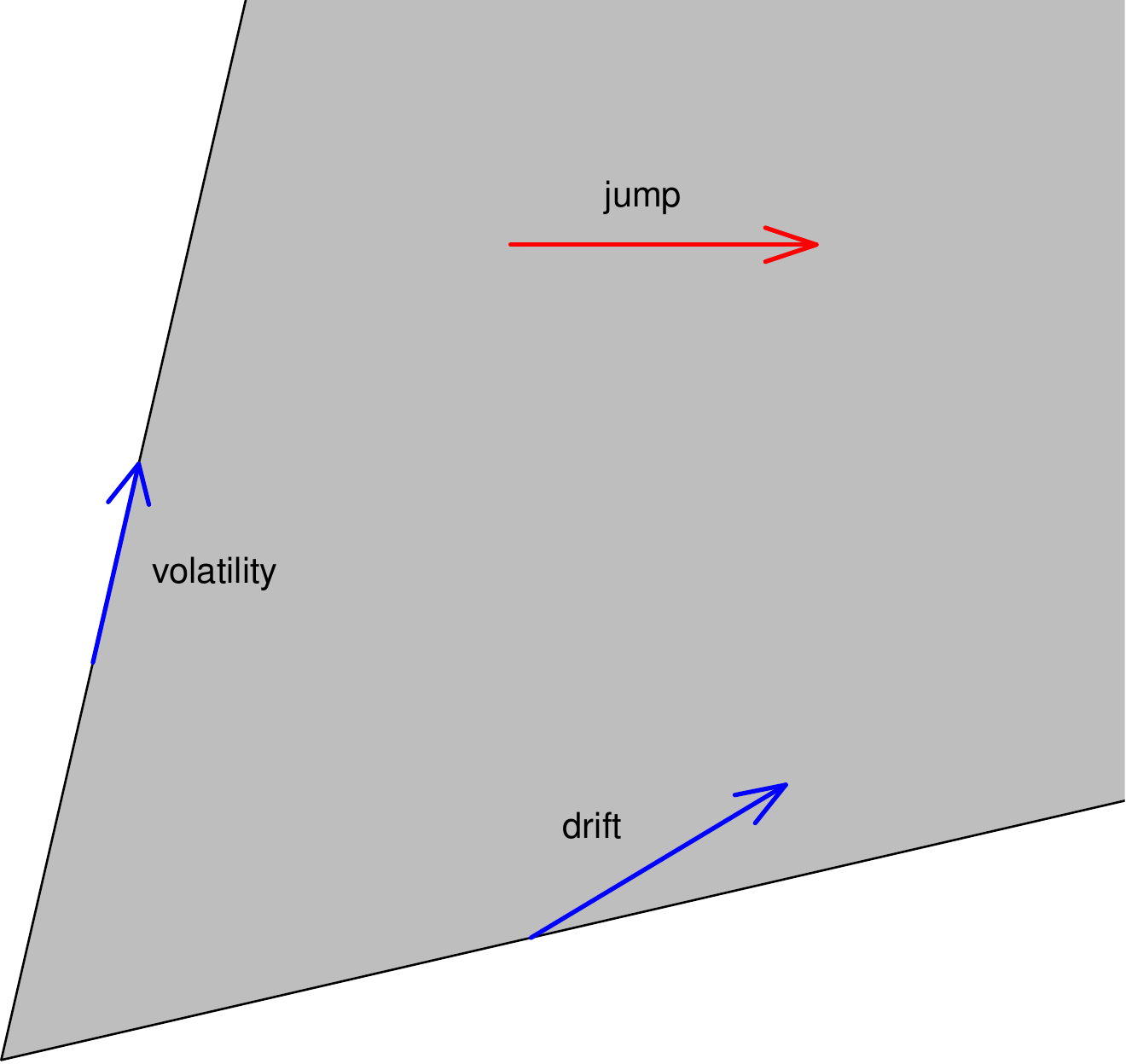}
 \caption{Illustration of the invariance conditions.}\label{fig-geometric}
\end{figure}
We refer to Section \ref{sec-ass} for the proofs of these and of further statements. We emphasize that for $(h^*,h) \in G^* \times K$ with $\langle h^*,h \rangle = 0$ it may happen that $(h^*,h) \notin D$. In this case, conditions (\ref{main-3}) -- and hence (\ref{main-2}) -- and (\ref{main-4}), the two boundary conditions illustrated in Figure \ref{fig-geometric}, do not need to be fulfilled. Intuitively, at such a boundary point $h$ of the cone, there is an infinite drift pulling the process in the interior of the half space $\{ h \in H : \langle h^*,h \rangle \geq 0 \}$, whence we can skip conditions (\ref{main-3}) and (\ref{main-4}) in this situation. This phenomenon is typical for SPDEs, as for norm continuous semigroups $(S_t)_{t \geq 0}$ (in particular, if $A = 0$) the limes inferior appearing in (\ref{def-D}) is always finite.

Now, let us outline the essential ideas for the proof of Theorem \ref{thm-main}:
\begin{itemize}
\item In Theorem \ref{thm-nec} we will prove that conditions (\ref{main-1})--(\ref{main-4}) are necessary for invariance of the cone $K$, where the main idea is to perform a short-time analysis of the sample paths of the solution processes. We emphasize that for this implication we do not need the assumption that $K$ is generated by an unconditional Schauder basis; that is, we can skip Assumption \ref{ass-Schauder-basis} here.

\item In order to show that conditions (\ref{main-1})--(\ref{main-4}) are sufficient for invariance of the cone $K$, we perform several steps:
\begin{enumerate}
\item First, we show that the cone $K$ is invariant for diffusion SPDEs (\ref{SPDE-Wiener}) with smooths volatilities $\sigma^j \in C_b^2(H)$, $j \in \bbn$; see Theorem \ref{thm-diffusion-C2}. The essential idea is to verify the aforementioned SSNC.

\item Then, we show that the cone $K$ is invariant for diffusion SPDEs (\ref{SPDE-Wiener}) with Lipschitz coefficients without imposing smoothness on the volatilities; see Theorem \ref{thm-diffusion}. The main idea is to approximate the volatility $\sigma$ by a sequence $(\sigma_n)_{n \in \bbn}$ of smooth volatilities, and to apply a stability result (see Proposition \ref{prop-K-stability}) for SPDEs.

\item Then, we show that the cone $K$ is invariant for general jump-diffusion SPDEs (\ref{SPDE}) with Lipschitz coefficients; see Theorem \ref{thm-suff}. This is done by using the so-called method to switch on the jumps -- also used in \cite{Positivity} -- and the aforementioned stability result for SPDEs.

\item Finally, we show that the cone $K$ is invariant for the SPDE (\ref{SPDE}) in the general situation, where the coefficients are locally Lipschitz and satisfy the linear growth condition; see Theorem \ref{thm-suff-general}. This is done by approximating the parameters $(\alpha,\sigma,\gamma)$ of the SPDE (\ref{SPDE}) by a sequence $(\alpha_n,\sigma_n,\gamma_n)_{n \in \bbn}$ of globally Lipschitz coefficients, and to argue by stability. In order to ensure that the modified coefficients $(\alpha_n,\sigma_n,\gamma_n)$ also satisfy the required invariance conditions (\ref{main-1})--(\ref{main-4}), the structural properties of closed convex cones are essential.
\end{enumerate}
\end{itemize}
The most challenging is the second step, where we approximate the volatility $\sigma$ by a sequence $(\sigma_n)_{n \in \bbn}$ of smooth volatilities. In particular, for an application of our stability result (Proposition \ref{prop-K-stability}) we must ensure that all $\sigma_n$ are Lipschitz continuous with a joint Lipschitz constant. We can roughly divide the approximation procedure into the following steps:
\begin{enumerate}
\item[(a)] First, we approximate $\sigma$ by a sequence $(\sigma_n)_{n \in \bbn}$ of bounded volatilities with finite dimensional range; see Propositions \ref{prop-sigma-FDR} and \ref{prop-sigma-bounded}. We construct similar approximations $(\alpha_n)_{n \in \bbn}$ for the drift $\alpha$; see Propositions \ref{prop-alpha-FDR} and \ref{prop-alpha-bounded}.

\item[(b)] Then, we approximate a bounded volatility $\sigma$ with finite dimensional range by a sequence $(\sigma_n)_{n \in \bbn}$ from $C_b^{1,1}$. This is done by the so-called sup-inf convolution technique from \cite{Lasry-Lions}; see Proposition \ref{prop-C-b-1-1}. Although we do not use it in this paper, we mention the related article \cite{Johanis}, which shows how a Lipschitz function can be approximated by uniformly G\^{a}teaux differentiable functions.

\item[(c)] Finally, we approximate a volatility $\sigma$ from $C_b^{1,1}$ by a sequence $(\sigma_n)_{n \in \bbn}$ from $C_b^2$; see Proposition \ref{prop-C-b-2}. This is done by a generalization of the mollifying technique in infinite dimension. For this procedure, we follow the construction provided in \cite{Fry}, which constitutes a generalization of a result from Moulis (see \cite{Moulis}), whence we also refer to this method as Moulis' method. Concerning smooth approximations in infinite dimensional spaces, we also mention the related papers \cite{Approx, Fine-Approx, Hajek-Johanis-G, Hajek-Johanis}.
\end{enumerate}
We emphasize that we cannot directly apply Moulis' method in step (b), because for a Lipschitz continuous function $\sigma$ this would only provide a sequence $(\sigma_n)_{n \in \bbn}$ from $C^2$ -- in fact, even $C^{\infty}$ -- but the second order derivatives might be unbounded. Applying the sup-inf convolution technique before ensures that we obtain a sequence from $C_b^2$. We mention that a combination of the sup-inf convolution technique and Moulis' method has also been used in \cite{Approx} in order to prove that every Lipschitz continuous function defined on a (possibly infinite dimensional) separable Riemannian manifold can be uniformly approximated by smooth Lipschitz functions.

Besides the aforementioned required joint Lipschitz constant, we have to take care that the respective approximations $(\sigma_n)_{n \in \bbn}$ of the volatility $\sigma$ remain parallel at boundary points of the cone; that is, condition (\ref{main-4}) must be preserved, which is expressed by Definition \ref{def-parallel}. The situation is similar for the approximations $(\alpha_n)_{n \in \bbn}$ of the drift $\alpha$. They must remain inward pointing at boundary points of the cone; that is, condition (\ref{main-3}) must be preserved, which is expressed by Definition \ref{def-inward-pointing}.

It arises the problem that we can generally not ensure in steps (b) and (c) that the approximating volatilities remain parallel. In order to illustrate the situation in step (c), where we apply Moulis' method, let us assume for the sake of simplicity that the state space is $H = \bbr^d$. Then the construction of the approximating sequence $(\sigma_n)_{n \in \bbn}$ becomes simpler than in the infinite dimensional situation in \cite{Fry}, and it is given by the well-known construction
\begin{align*}
\sigma_n : \bbr^d \to \bbr^d, \quad \sigma_n(h) := \int_{\bbr^d} \sigma(h-g) \varphi_n(g) dg,
\end{align*}
where $(\varphi_n)_{n \in \bbn} \subset C^{\infty}(\bbr^d,\bbr_+)$ is an appropriate sequence of mollifiers. Then, for $(h^*,h) \in D$, which implies $\langle h^*,h \rangle = 0$, we generally have
\begin{align*}
\langle h^*,\sigma_n(h) \rangle = \int_{\bbr^d} \langle h^*, \sigma(h-g) \rangle \varphi_n(g) dg \neq 0,
\end{align*}
because we only have $\langle h^*, \sigma(h) \rangle = 0$, but generally not $\langle h^*, \sigma(h-g) \rangle = 0$ for all $g \in \bbr^d$ from a neighborhood of $0$. This problem leads to the notion of locally parallel functions (see Definition \ref{def-locally-par}), which have the desired property that  $\langle h^*, \sigma(h-g) \rangle = 0$ for all $g \in \bbr^d$ from an appropriate neighborhood of $0$. In order to implement this concept, we have to show that a parallel function can be approximated by a sequence of locally parallel functions. The idea is to approximate a function $\sigma : \bbr^d \to \bbr^d$ for $\epsilon > 0$ by taking $\sigma \circ \Phi_{\epsilon}$, where
\begin{align*}
\Phi_{\epsilon} : \bbr^d \to \bbr^d, \quad \Phi_{\epsilon}(h) := (\phi_{\epsilon}(h_1),\ldots,\phi_{\epsilon}(h_d)),
\end{align*}
and where the function $\phi_{\epsilon} : \bbr \to \bbr$ is defined as
\begin{align}\label{def-psi-intro}
\phi_{\epsilon}(x) := (x + \epsilon) \mathbbm{1}_{(-\infty,-\epsilon]}(x) + (x - \epsilon) \mathbbm{1}_{[\epsilon,\infty)}(x),
\end{align}
see Figure \ref{fig-approx}. We can also establish this procedure in infinite dimension; see Proposition \ref{prop-locally-par}.

\begin{figure}[!ht]
 \centering
 \includegraphics[width=0.5\textwidth]{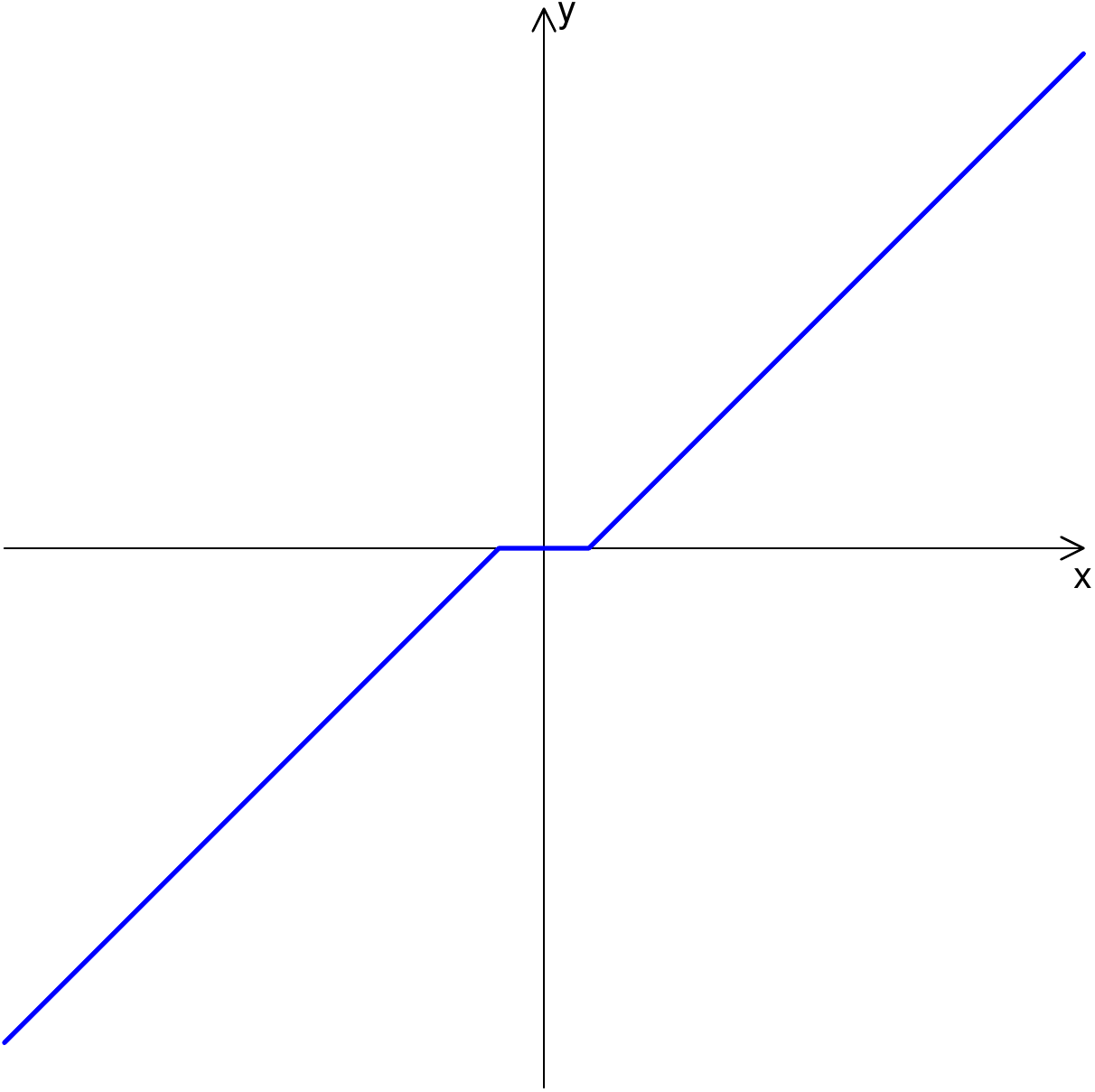}
 \caption{Approximation with locally parallel functions.}\label{fig-approx}
\end{figure}

The remainder of this paper is organized as follows. In Section \ref{sec-ass} we present the mathematical framework and preliminary results. In Section \ref{sec-nec} we prove that our invariance conditions are necessary for invariance of the cone. In Section \ref{sec-Schauder} we provide the required background about closed convex cones generated by unconditional Schauder basis. Afterwards, we start with the proof that our invariance conditions are sufficient for invariance of in the cone. In Section \ref{sec-proof-1} we prove this for diffusion SPDEs with smooth volatilities, in Section \ref{sec-proof-2} for diffusion SPDEs with Lipschitz coefficients without imposing smoothness on the volatility, in Section \ref{sec-proof-3} for general jump-diffusion SPDEs with Lipschitz coefficients, and in Section \ref{sec-proof-4} for the general situation of jump-diffusion SPDEs with coefficients being locally Lipschitz and satisfying the linear growth condition. In Section \ref{sec-example} we provide an example illustrating our main result. In Appendix \ref{app-function-spaces} we collect the function spaces which we use throughout this paper, and in Appendix \ref{app-stability} we present the required stability result for SPDEs. In Appendix \ref{app-drift} we provide the required results about inward pointing functions, and in Appendix \ref{app-volatility} about parallel functions.

\section{Mathematical framework and preliminary results}\label{sec-ass}

In this section, we present the mathematical framework and preliminary results. Let $(\Omega,\calf,(\calf_t)_{t \in \bbr_+},\bbp)$ be a filtered probability space satisfying the usual conditions. Let $H$ be a separable Hilbert space and let $A : \cald(A) \subset H \to H$ be the infinitesimal generator of a $C_0$-semigroup $(S_t)_{t \geq 0}$ on $H$. 

\begin{assumption}\label{ass-pseudo-contractive}
We assume that the semigroup $(S_t)_{t \geq 0}$ is pseudo-contractive; that is, there exists a constant $\beta \geq 0$ such that
\begin{align}\label{growth-semigroup}
\| S_t \| \leq e^{\beta t} \quad \text{for all $t \geq 0$.}
\end{align}
\end{assumption}

In view of condition (\ref{growth-semigroup}), we emphasize that for $h \in H$ we denote by $\| h \|$ the Hilbert space norm, and that for a bounded linear operator $T \in L(H)$ we denote by $\| T \|$ the operator norm
\begin{align*}
\| T \| = \inf \{ M \geq 0 : \| Th \| \leq M \| h \| \text{ for all } h \in H \}.
\end{align*}
Let $U$ be a separable Hilbert space, and let $W$ be an $U$-valued $Q$-Wiener process for some nuclear, self-adjoint, positive definite linear operator $Q \in L(U)$; see \cite[pages 86, 87]{Da_Prato}. There exist an orthonormal basis $\{ e_j \}_{j \in \bbn}$ of $U$ and a sequence $(\lambda_j)_{j \in \bbn} \subset (0,\infty)$ with $\sum_{j \in \bbn} \lambda_j < \infty$ such that
\begin{align*}
Q e_j = \lambda_j e_j \quad \text{for all $j \in \bbn$.}
\end{align*}
Let $(E,\cale)$ be a Blackwell space, and let $\mu$ be a homogeneous Poisson random measure with compensator $dt \otimes F(dx)$ for some $\sigma$-finite measure $F$ on $(E,\cale)$; see \cite[Def. II.1.20]{Jacod-Shiryaev}.
The space $U_0 := Q^{1/2}(U)$, equipped with the inner product
\begin{align}\label{inner-prod-U0}
\langle u,v \rangle_{U_0} := \langle Q^{-1/2}u, Q^{-1/2}v \rangle_U,
\end{align}
is another separable Hilbert space. We denote by $L_2^0(H) := L_2(U_0,H)$ the space of all Hilbert-Schmidt operators from $U_0$ into $H$. We fix the orthonormal basis $\{ g_j \}_{j \in \bbn}$ of $U_0$ given by $g_j := \sqrt{\lambda_j} e_j$ for each $j \in \bbn$, and for each $\sigma \in L_2^0(H)$ we set $\sigma^j := \sigma g_j$ for $j \in \bbn$. Furthermore, we denote by $L^2(F) := L^2(E,\cale,F;H)$ the space of all square-integrable functions from $E$ into $H$.
Let $\alpha : H \to H$, $\sigma : H \to L_2^0(H)$ and $\gamma : H \to L^2(F)$ be measurable functions. Concerning the upcoming notation, we remind the reader that in Appendix \ref{app-function-spaces} we have collected the function spaces used in this paper.

\begin{assumption}\label{ass-loc-Lip-LG}
We suppose that
\begin{align*}
\alpha &\in \Lip^{\loc}(H) \cap \LG(H),
\\ \sigma &\in \Lip^{\loc}(H,L_2^0(H)) \cap \LG(H,L_2^0(H)),
\\ \gamma &\in \Lip^{\loc}(H,L^2(F)) \cap \LG(H,L^2(F)).
\end{align*}
\end{assumption}

Assumption \ref{ass-loc-Lip-LG} ensures that for each $h_0 \in H$ the SPDE (\ref{SPDE}) has a unique mild solution; that is, an $H$-valued c\`{a}dl\`{a}g adapted process $r$, unique up to indistinguishability, such that
\begin{equation}\label{mild-solution}
\begin{aligned}
r_t &= S_t h_0 + \int_0^t S_{t-s} \alpha(r_s) ds + \int_0^t S_{t-s} \sigma(r_s) dW_s
\\ &\quad + \int_0^t S_{t-s} \gamma(r_{s-},x) (\mu(ds,dx) - F(dx)ds), \quad t \in \bbr_+.
\end{aligned}
\end{equation}
The sequence $(\beta^j)_{j \in \bbn}$ defined as
\begin{align}\label{beta-j}
\beta^j := \frac{1}{\sqrt{\lambda_j}} \langle W,e_j \rangle, \quad j \in \bbn
\end{align}
is a sequence of real-valued standard Wiener processes, and we can write (\ref{mild-solution}) equivalently as
\begin{equation}\label{mild-solution-beta}
\begin{aligned}
r_t &= S_t h_0 + \int_0^t S_{t-s} \alpha(r_s) ds + \sum_{j \in \bbn} \int_0^t S_{t-s} \sigma^j(r_s) d\beta_s^j
\\ &\quad + \int_0^t S_{t-s} \gamma(r_{s-},x) (\mu(ds,dx) - F(dx)ds), \quad t \in \bbr_+.
\end{aligned}
\end{equation}
Note that Assumption \ref{ass-loc-Lip-LG} is implied by the slightly stronger conditions
\begin{align*}
\alpha \in \Lip(H), \quad \sigma &\in \Lip(H,L_2^0(H)) \quad \text{and} \quad \gamma \in \Lip(H,L^2(F)).
\end{align*}
Under such global Lipschitz conditions, we refer the reader to \cite{Da_Prato, Prevot-Roeckner, Atma-book, Liu-Roeckner} for diffusion SPDEs, to \cite{P-Z-book} for L\'{e}vy driven SPDEs, and to \cite{MPR, SPDE} for general jump-diffusion SPDEs. Under the local Lipschitz and linear growth conditions from Assumption \ref{ass-loc-Lip-LG}, we refer to \cite{Tappe-refine}.

\begin{definition}
A subset $K \subset H$ is called \emph{invariant} for the SPDE (\ref{SPDE}) if for each $h_0 \in K$ we have $r \in K$ up to an evanescent set\footnote[1]{A random set $A \subset \Omega \times \mathbb{R}_+$ is called \emph{evanescent} if the set $\{ \omega \in \Omega : (\omega,t) \in A \text{ for some } t \in \mathbb{R}_+ \}$ is a $\mathbb{P}$-nullset, cf. \cite[1.1.10]{Jacod-Shiryaev}.}, where $r$ denotes the mild solution to (\ref{SPDE}) with $r_0 = h_0$.
\end{definition}

\begin{definition}\label{def-cone}
A subset $K \subset H$ is called a \emph{cone} if we have $\lambda h \in K$ for all $\lambda \geq 0$ and all $h \in K$.
\end{definition}

\begin{definition}\label{def-conv-cone}
A cone $K \subset H$ is called a \emph{convex cone} if we have $h+g \in K$ for all $h,g \in H$.
\end{definition}

Note that a convex cone $K \subset H$ is indeed a convex subset of $H$.

\begin{definition}
A convex cone $K \subset H$ is called a \emph{closed convex cone} if it is closed as a subset of $H$.
\end{definition}

For what follows, we fix a closed convex cone $K \subset H$. Denoting by $K^* \subset H$ its dual cone (\ref{dual-cone}), the cone $K$ has the representation (\ref{cone-repr}).

\begin{definition}
A subset $G^* \subset K^*$ is called a \emph{generating system} of the cone $K$ if we have the representation (\ref{cone-G}).
\end{definition}

Of course $G^* = K^*$ is a generating system of the cone $K$. However, for applications we will choose the generating system $G^*$ as convenient as possible. In this respect, we mention that, by Lindel\"{o}f's lemma, the cone $K$ admits a generating system $G^*$ which is at most countable. For what follows, we fix a generating system $G^* \subset K^*$.

\begin{definition}
For a function $f : H \to H$ we say that $K$ is \emph{$f$-invariant} if $f(K) \subset K$.
\end{definition}

\begin{definition}
The closed convex cone $K$ is called \emph{invariant} for the semigroup $(S_t)_{t \geq 0}$ if $K$ is $S_t$-invariant for all $t \geq 0$.
\end{definition}

According to \cite[Cor. 1.10.6]{Pazy} the adjoint semigroup $(S_t^*)_{t \geq 0}$ is a $C_0$-semigroup on $H$ with infinitesimal generator $A^*$.

\begin{lemma}\label{lemma-cone-adj}
The following statements are equivalent:
\begin{enumerate}
\item[(i)] $K$ is invariant for the semigroup $(S_t)_{t \geq 0}$.

\item[(ii)] $K^*$ is invariant for the adjoint semigroup $(S_t^*)_{t \geq 0}$.
\end{enumerate}
\end{lemma}

\begin{proof}
For all $(h^*,h) \in K^* \times K$ and all $t \geq 0$ we have
\begin{align*}
\langle h^*,S_t h \rangle = \langle S_t^* h^*,h \rangle,
\end{align*}
and hence, the representations (\ref{cone-repr}) and (\ref{dual-cone}) of $K$ and $K^*$ prove the claimed equivalence.
\end{proof}

For $\lambda > \beta$, where the constant $\beta \geq 0$ stems from the growth estimate (\ref{growth-semigroup}), we define the resolvent $R_{\lambda} := (\lambda - A)^{-1}$. We consider the abstract Cauchy problem
\begin{align}\label{PDE-Cauchy}
\left\{
\begin{array}{rcl}
dr_t & = & A r_t dt \medskip
\\ r_0 & = & h_0.
\end{array}
\right.
\end{align}

\begin{lemma}\label{lemma-cone-semi-inv}
The following statements are equivalent:
\begin{enumerate}
\item[(i)] $K$ is invariant for the semigroup $(S_t)_{t \geq 0}$.

\item[(ii)] $K$ is invariant for the abstract Cauchy problem (\ref{PDE-Cauchy}).

\item[(iii)] $K$ is $R_{\lambda}$-invariant for all $\lambda > \beta$.
\end{enumerate}
\end{lemma}

\begin{proof}
(i) $\Leftrightarrow$ (ii): This equivalence follows, because for each $h_0 \in K$ the mild solution to the abstract Cauchy problem (\ref{PDE-Cauchy}) is given by $r_t = S_t h_0$ for $t \geq 0$.
\\ \noindent(i) $\Rightarrow$ (iii): For each $\lambda > \beta$ and each $h \in K$ we have
\begin{align*}
R_{\lambda} h = \int_0^{\infty} e^{-\lambda t} S_t h \, dt \in K.
\end{align*}
(iii) $\Rightarrow$ (i): Let $t > 0$ and $h \in K$ be arbitrary. By the exponential formula (see \cite[Thm. 1.8.3]{Pazy}) we have
\begin{align*}
S_t h = \lim_{n \to \infty} \bigg( \frac{n}{t} R_{n/t} \bigg)^n h \in K,
\end{align*}
completing the proof.
\end{proof}

From now on, we make the following assumption.

\begin{assumption}\label{ass-cone-semigroup}
We assume that the cone $K$ is invariant for the semigroup $(S_t)_{t \geq 0}$; that is, any of the equivalent conditions from Lemma \ref{lemma-cone-semi-inv} is fulfilled.
\end{assumption}

\begin{lemma}
For all $(h^*,h) \in G^* \times K$ we have
\begin{align*}
\liminf_{t \downarrow 0} \frac{\langle h^*, S_t h \rangle}{t} \in \overline{\bbr}_+.
\end{align*}
\end{lemma}

\begin{proof}
Since $K$ is invariant for the semigroup $(S_t)_{t \geq 0}$, we have $\langle h^*,S_t h \rangle \geq 0$ for all $t \geq 0$, which establishes the proof.
\end{proof}

\begin{definition}
For $g,h \in H$ we write $g \leq_K h$ if $h-g \in K$.
\end{definition}

Recall the set $D \subset G^* \times K$ defined in (\ref{def-D}). We define the function
\begin{align*}
a : D \to \bbr_+, \quad a(h^*,h) := \liminf_{t \downarrow 0} \frac{\langle h^*,S_t h \rangle}{t}.
\end{align*}

\begin{lemma}\label{lemma-fct-a}
For each $(h^*,h) \in D$ the following statements are true:
\begin{enumerate}
\item We have $\langle h^*,h \rangle = 0$.

\item For all $\lambda \geq 0$ we have $(h^*,\lambda h) \in D$ and
\begin{align}\label{id-a}
a(h^*,\lambda h) = \lambda a(h^*,h).
\end{align}
\item For all $g \in K$ with $g \leq_K h$ we have $(h^*,g) \in D$ and
\begin{align}\label{id-a-g}
a(h^*,g) \leq a(h^*,h).
\end{align}
\end{enumerate}
\end{lemma}

\begin{proof}
For each $(h^*,h) \in G^* \times K$ with $\langle h^*,h \rangle > 0$ we have
\begin{align*}
\lim_{t \downarrow 0} \langle h^*,S_t h \rangle = \langle h^*,h \rangle > 0,
\end{align*}
and hence
\begin{align*}
\liminf_{t \downarrow 0} \frac{\langle h^*,S_t h \rangle}{t} = \infty,
\end{align*}
showing that $(h^*,h) \notin D$. This proves the first statement, and we proceed with the second statement.
Since $K$ is a cone, we have $\lambda h \in K$. Furthermore, we have
\begin{align*}
\liminf_{t \downarrow 0} \frac{\langle h^*,S_t(\lambda h) \rangle}{t} = \lambda \liminf_{t \downarrow 0} \frac{\langle h^*,S_t h \rangle}{t} < \infty,
\end{align*}
showing $(h^*,\lambda h) \in D$ and the identity (\ref{id-a}). For the proof of the third statement, let $t \geq 0$ be arbitrary. By Lemma \ref{lemma-cone-adj} we have $S_t^* h^* \in K^*$. Since $g \leq_K h$, we obtain $\langle S_t^* h^*,h-g \rangle \geq 0$, and hence
\begin{align*}
\langle h^*,S_t g \rangle = \langle S_t^* h^*,g \rangle \leq \langle S_t^* h^*,h \rangle = \langle h^*,S_t h \rangle.
\end{align*}
Consequently, we have
\begin{align}\label{compare-semi}
\langle h^*,S_t g \rangle \leq \langle h^*,S_t h \rangle \quad \text{for all $t \geq 0$.}
\end{align}
There exists a sequence $(t_n)_{n \in \bbn} \subset (0,\infty)$ with $t_n \downarrow 0$ such that the sequence $(b_n)_{n \in \bbn} \subset \bbr_+$ defined as
\begin{align*}
b_n := \frac{\langle h^*,S_{t_n} h \rangle}{t_n}, \quad n \in \bbn 
\end{align*}
converges to $a(h^*,h) \in \bbr_+$. Defining the sequence $(a_n)_{n \in \bbn} \subset \bbr_+$ as
\begin{align*}
a_n := \frac{\langle h^*,S_{t_n} g \rangle}{t_n}, \quad n \in \bbn,
\end{align*}
by (\ref{compare-semi}) we have $0 \leq a_n \leq b_n$ for each $n \in \bbn$. Hence, the sequence $(a_n)_{n \in \bbn}$ is bounded, and by the Bolzano-Weierstrass theorem there exists a subsequence $(n_k)_{k \in \bbn}$ such that $(a_{n_k})_{k \in \bbn}$ converges to some $a \in \bbr_+$ with $a \leq a(h^*,h)$, which proves $(h^*,g) \in D$ and (\ref{id-a-g}).
\end{proof}

\begin{lemma}\label{lemma-liminf-domain}
Let $(h^*,h) \in G^* \times K$ with $\langle h^*,h \rangle = 0$ be arbitrary. Then the following statements are true:
\begin{enumerate}
\item If $h \in \cald(A)$, then we have $(h^*,h) \in D$ and
\begin{align}\label{liminf-domain}
\liminf_{t \downarrow 0} \frac{\langle h^*, S_t h \rangle}{t} = \langle h^*,A h \rangle.
\end{align}
\item If $h^* \in \cald(A^*)$, then we have $(h^*,h) \in D$ and
\begin{align}\label{liminf-domain-2}
\liminf_{t \downarrow 0} \frac{\langle h^*, S_t h \rangle}{t} = \langle A^* h^*,h \rangle.
\end{align}
\item If the semigroup $(S_t)_{t \geq 0}$ is norm continuous, then we have $(h^*,h) \in D$ as well as (\ref{liminf-domain}) and (\ref{liminf-domain-2}).
\end{enumerate}
\end{lemma}

\begin{proof}
If $h \in \cald(A)$, then we have
\begin{align*}
\frac{\langle h^*, S_t h \rangle}{t} = \frac{\langle h^*, S_t h \rangle - \langle h^*,h \rangle}{t} = \frac{\langle h^*, S_t h - h \rangle}{t} = \bigg\langle h^*, \frac{S_t h - h}{t} \bigg\rangle \to \langle h^*,A h \rangle
\end{align*}
as $t \downarrow 0$, showing the first statement. Furthermore, if $h^* \in \cald(A^*)$, then we obtain
\begin{align*}
\frac{\langle h^*,S_t h \rangle}{t} &= \frac{\langle S_t^* h^*,h \rangle}{t} = \frac{\langle S_t^* h^*,h \rangle - \langle h^*,h \rangle}{t} 
\\ &= \frac{\langle S_t^* h^* - h^*,h \rangle}{t} = \bigg\langle \frac{S_t^* h^* - h^*}{t}, h \bigg\rangle \to \langle A^* h^*,h \rangle
\end{align*}
as $t \downarrow 0$, showing the second statement. The third statement is an immediate consequence of the first and the second statement.
\end{proof}

The following definition is inspired by \cite[Lemma 5]{Milian}.

\begin{definition}\label{def-local-op}
We call $A^*$ a \emph{local operator} if $G^* \subset \cald(A^*)$, and for all $(h^*,h) \in D$ we have $\langle A^* h^*,h \rangle = 0$.
\end{definition}

\begin{proposition}\label{prop-conditions}
Suppose that condition (\ref{main-1}) is fulfilled. Then for all $(h^*,h) \in D$ the following statements are true:
\begin{enumerate}
\item We have
\begin{align*}
\langle h^*,\gamma(h,x) \rangle \geq 0 \quad \text{for $F$-almost all $x \in E$.}
\end{align*}

\item We have
\begin{align*}
\int_E \langle h^*,\gamma(h,x) \rangle F(dx) \in \overline{\bbr}_+.
\end{align*}

\item If condition (\ref{main-3}) is satisfied, then we have (\ref{main-2}).

\item If $h \in \cald(A)$, then conditions (\ref{main-3}) and (\ref{FV-cond-3}) are equivalent.

\item If $h^* \in \cald(A^*)$, then conditions (\ref{main-3}) and (\ref{FV-cond-4}) are equivalent.

\item If $A^*$ is a local operator, then conditions (\ref{main-3}) and (\ref{FV-cond-5}) are equivalent.

\item Condition (\ref{FV-cond-5}) implies (\ref{main-3}).
\end{enumerate}
\end{proposition}

\begin{proof}
By (\ref{main-1}), for $F$-almost all $x \in E$ we have
\begin{align*}
\langle h^*,\gamma(h,x) \rangle = \langle h^*,h \rangle + \langle h^*,\gamma(h,x) \rangle = \langle h^*,h+\gamma(h,x) \rangle \geq 0,
\end{align*}
which establishes the first statement. The second statement is an immediate consequence, and the third statement is obvious. The fourth and the fifth statement follow from Lemma \ref{lemma-liminf-domain}. Taking into account Definition \ref{def-local-op}, the sixth statement is an immediate consequence of the fifth statement. Finally, the last statement follows from the first statement. 
\end{proof}

In view of condition (\ref{FV-cond-3}), we emphasize that $K \cap \cald(A)$ is dense is $K$, which follows from the next result.

\begin{lemma}\label{lemma-cone-dense}
We have $K = \overline{K \cap \cald(A)}$.
\end{lemma}

\begin{proof}
Since $K$ is closed, we have $\overline{K \cap \cald(A)} \subset K$. In order to prove the converse inclusion, let $h \in K$ be arbitrary. For $t > 0$ we set $h_t := \frac{1}{t} \int_0^t S_s h ds$. Then we have $h_t \in \cald(A)$ for each $t > 0$, and we have $h_t \to h$ for $t \downarrow 0$. It remains to show that $h_t \in K$ for each $t > 0$. For this purpose, let $t > 0$ and $h^* \in G^*$ be arbitrary. Since $K$ is invariant for the semigroup $(S_t)_{t \geq 0}$, we obtain
\begin{align*}
\langle h^*,h_t \rangle = \bigg\langle h^*,\frac{1}{t} \int_0^t S_s h ds \bigg\rangle = \frac{1}{t} \int_0^t \langle h^*,S_s h \rangle ds \geq 0,
\end{align*}
showing that $h_t \in K$.
\end{proof}

\section{Necessity of the invariance conditions}\label{sec-nec}

In this section, we prove the necessity of our invariance conditions.

\begin{theorem}\label{thm-nec}
Suppose that Assumptions \ref{ass-pseudo-contractive}, \ref{ass-loc-Lip-LG} and \ref{ass-cone-semigroup} are fulfilled. If the closed convex cone $K$ is invariant for the SPDE (\ref{SPDE}), then we have (\ref{main-1}), and for all $(h^*,h) \in D$ we have (\ref{main-3}) and (\ref{main-4}).
\end{theorem}

\begin{proof}
Condition (\ref{main-1}) follows from \cite[Lemma 2.11]{FTT-appendix}. Let $(h^*,h) \in D$ be arbitrary, and denote by $r$ the mild solution to (\ref{SPDE}) with $r_0 = h$. Since the measure space $(E,\cale,F)$ is $\sigma$-finite, there exists an increasing sequence $(B_n)_{n \in \bbn} \subset \cale$ with $F(B_n) < \infty$ for each $n \in \bbn$ such that $E = \bigcup_{n \in \bbn} B_n$. Let $n \in \bbn$ be arbitrary. According to \cite[Lemma 2.20]{FTT-appendix} the mapping $T_n : \Omega \to \overline{\bbr}_+$ given by
\begin{align*}
T_n := \inf \{ t \in \bbr_+ : \mu([0,t] \times B_n) = 1 \}
\end{align*}
is a strictly positive stopping time. We denote by $r^n$ the mild solution to the SPDE
\begin{align*}
\left\{
\begin{array}{rcl}
dr_t^n & = & (A r_t^n + \alpha(r_t^n) - \int_{B_n} \gamma(r_t^n,x) F(dx)) dt + \sigma(r_t^n) dW_t
\\ && + \int_{B_n^c} \gamma(r_{t-}^n,x) (\mu(dt,dx) - F(dx) dt) \medskip
\\ r_0^n & = & h.
\end{array}
\right.
\end{align*}
Since $K$ is a closed subset of $H$, by \cite[Prop. 2.21]{FTT-appendix} we obtain $(r^n)^{T_n} \in K$ up to an evanescent set. We define the strictly positive, bounded stopping time
\begin{align*}
T &:= \inf \{ t \in \bbr_+ : \| r_t^n \| > 1 + \| h \| \} \wedge T_n \wedge 1.
\end{align*}
Furthermore, for every stopping time $R \leq T$ we define the processes $A^n(R)$ and $M^n(R)$ as
\begin{align*}
A^n(R)_t &:= \int_0^t \bigg\langle h^*,S_{R-s} \bigg( \alpha(r_s^n) - \int_{B_n} \gamma(r_s^n,x) F(dx) \bigg) \bigg\rangle \bbI_{\{ R \geq s \}} ds, \quad t \in \bbr_+,
\\ M^n(R)_t &:= \int_0^t \langle h^*,S_{R-s} \sigma(r_s^n) \rangle \bbI_{\{ R \geq s \}} dW_s
\\ &\quad + \int_0^t \int_{B_n} \langle h^*,S_{R-s} \gamma(r_{s-}^n,x) \rangle \bbI_{\{ R \geq s \}} (\mu(ds,dx) - F(dx)ds), \quad t \in \bbr_+.
\end{align*}
Then, by the Cauchy-Schwarz inequality and Assumptions \ref{ass-pseudo-contractive}, \ref{ass-loc-Lip-LG} we have $A^n(R) \in \cala$ and $M^n(R) \in \calh^2$ for each stopping time $R \leq T$, where $\cala$ denotes the space of all finite variation processes with integrable variation (see \cite[I.3.7]{Jacod-Shiryaev}) and $\calh^2$ denotes the space of all square-integrable martingales (see \cite[Def. I.1.41]{Jacod-Shiryaev}). Moreover, we have $\bbp$-almost surely
\begin{align*}
0 \leq \langle h^*,r_{T \wedge t}^n \rangle = \langle h^*,S_{T \wedge t} h \rangle + A^n(T \wedge t)_{T \wedge t} + M^n(T \wedge t)_{T \wedge t} \quad \text{for all $t \in \bbr_+$.}
\end{align*}
Let $(t_k)_{k \in \bbn} \subset (0,\infty)$ be a sequence with $t_k \downarrow 0$ such that
\begin{align}\label{liminf-nec}
\liminf_{t \downarrow 0} \frac{\langle h^*,S_t h \rangle}{t} =
\lim_{k \to \infty} \frac{\langle h^*,S_{t_k} h \rangle}{t_k}.
\end{align}
By Lebesgue's dominated convergence theorem we obtain
\begin{align*}
0 &\leq \lim_{k \to \infty} \frac{1}{t_k} \bbe[ \langle h^*,r_{T \wedge t_k}^n \rangle ] = \lim_{k \to \infty} \frac{1}{t_k} \bbe [\langle h^*,S_{T \wedge t_k} h \rangle] + \lim_{k \to \infty} \frac{1}{t_k} \bbe[A^n(T \wedge t_k)_{T \wedge t_k}]
\\ &= \lim_{k \to \infty} \frac{\langle h^*,S_{t_k} h \rangle}{t_k} + \langle h^*,\alpha(h) \rangle - \int_{B_n} \langle h^*,\gamma(h,x) \rangle F(dx),
\end{align*}
showing that
\begin{align}\label{drift-nec-1}
\liminf_{t \downarrow 0} \frac{1}{t} \langle h^*,S_t h \rangle + \langle h^*,\alpha(h) \rangle - \int_{B_n} \langle h^*,\gamma(h,x) \rangle F(dx) \geq 0.
\end{align}
Furthermore, by the monotone convergence theorem and Proposition \ref{prop-conditions} we have
\begin{align}\label{drift-nec-2}
\int_E \langle h^*,\gamma(h,x) \rangle F(dx) = \lim_{n \to \infty} \int_{B_n} \langle h^*,\gamma(h,x) \rangle F(dx).
\end{align}
Combining (\ref{drift-nec-1}) and (\ref{drift-nec-2}), we arrive at (\ref{main-3}).

Now, suppose that condition (\ref{main-4}) is not fulfilled. Then there exist $j \in \bbn$ and $(h^*,h) \in D$ such that $\langle h^*,\sigma^j(h) \rangle \neq 0$. We define $\eta,\Phi \in \bbr$ by
\begin{align}\label{def-eta}
\eta := \liminf_{t \downarrow 0} \frac{\langle h^*,S_t h \rangle}{t} + \langle h^*,\alpha(h) \rangle \quad \text{and} \quad \Phi := -\frac{\eta + 1}{\langle h^*,\sigma^j(h) \rangle}.
\end{align}
Note that, by (\ref{main-3}) and Proposition \ref{prop-conditions} we have $\eta \in \bbr_+$. The stochastic exponential
\begin{align*}
Z := \cale(\Phi \beta^j),
\end{align*}
where the Wiener process $\beta^j$ is given by (\ref{beta-j}), is a strictly positive, continuous local martingale. We define the strictly positive, bounded stopping time
\begin{align*}
T &:= \inf \{ t \in \bbr_+ : \| r_t \| > 1 + \| h \| \} \wedge \inf \{ t \in \bbr_+ : | Z_t | > 2 \}
\\ &\quad \wedge \inf \{ t \in \bbr_+ : \langle Z,Z \rangle_t > 1 \} \wedge 1.
\end{align*} 
For every stopping time $R \leq T$ we define the processes $A(R)$, $M(R)$ and $N(R)$ as
\begin{align*}
A(R)_t &:= \int_0^t \langle h^*,S_{R-s} \alpha(r_s) \rangle \bbI_{\{ R \geq s \}} ds, \quad t \in \bbr_+,
\\ M(R)_t &:= \int_0^t \langle h^*,S_{R-s} \sigma(r_s) \rangle \bbI_{\{ R \geq s \}} dW_s
\\ &\quad + \int_0^t \int_E \langle h^*,S_{R-s} \gamma(r_{s-},x) \rangle \bbI_{\{ R \geq s \}} (\mu(ds,dx) - F(dx)ds), \quad t \in \bbr_+,
\\ N(R)_t &:= \int_0^t ( A(R)_{s-} + M(R)_{s-} ) \bbI_{\{ R \geq s \}} dZ_s + \int_0^t Z_s \bbI_{\{ R \geq s \}} dM(R)_s, \quad t \in \bbr_+.
\end{align*}
Then, by Assumptions \ref{ass-pseudo-contractive}, \ref{ass-loc-Lip-LG} we have $A(R) \in \cala$ and $M(R),N(R) \in \calh^2$ for each stopping time $R \leq T$. Moreover, we have $\bbp$-almost surely
\begin{align*}
0 \leq \langle h^*,r_{T \wedge t} \rangle = \langle h^*,S_{T \wedge t} h \rangle + A(T \wedge t)_{T \wedge t} + M(T \wedge t)_{T \wedge t} \quad \text{for all $t \in \bbr_+$.}
\end{align*}
Let $R \leq T$ be an arbitrary stopping time. By \cite[Prop. I.4.49]{Jacod-Shiryaev} we have $[A(R),Z^R] = 0$, and by \cite[Thm. I.4.52]{Jacod-Shiryaev} we have $[M(R),Z^R] = \langle M(R)^c,Z^R \rangle$. Therefore, and since
\begin{align*}
Z_t^R = 1 + \Phi \int_0^t Z_s \bbI_{\{ R \geq s \}} d\beta_s^j, \quad t \in \bbr_+,
\end{align*}
by \cite[Def. I.4.45]{Jacod-Shiryaev} we obtain
\begin{equation}\label{int-by-parts}
\begin{aligned}
&( A(R)_t + M(R)_t ) Z_t^R = N(R)_t + \int_0^t Z_s \bbI_{\{ R \geq s \}} dA(R)_s + \langle M(R)^c,Z^R \rangle
\\ &= N(R)_t + \int_0^t \langle h^*, S_{R-s} ( \alpha(r_s) + \Phi \sigma^j(r_s) ) \rangle Z_s \bbI_{\{ R \geq s \}} ds, \quad t \in \bbr_+.
\end{aligned}
\end{equation}
Let $(t_k)_{k \in \bbn} \subset (0,\infty)$ be a sequence with $t_k \downarrow 0$ such that we have (\ref{liminf-nec}). By (\ref{int-by-parts}), Lebesgue's dominated convergence theorem and (\ref{def-eta}) we obtain
\begin{align*}
0 &\leq \lim_{k \to \infty} \frac{1}{t_k} \bbe[ \langle h^*,r_{T \wedge t_k}^n \rangle Z_{T \wedge t_k} ] = \lim_{k \to \infty} \frac{1}{t_k} \bbe [\langle h^*,S_{T \wedge t_k} h \rangle Z_{T \wedge t_k}]
\\ &\quad + \lim_{k \to \infty} \frac{1}{t_k} \bbe[(A(T \wedge t_k)_{T \wedge t_k} + M(T \wedge t_k)_{T \wedge t_k} )Z_{T \wedge t_k}^{T \wedge t_k}]
\\ &= \liminf_{t \downarrow 0} \frac{\langle h^*,S_t h \rangle}{t} + \langle h^*,\alpha(h) + \Phi \sigma^j(h) \rangle
\\ &= \eta + \Phi \langle h^*,\sigma^j(h) \rangle = \eta - (\eta + 1) = -1,
\end{align*}
a contradiction.
\end{proof}

\section{Cones generated by unconditional Schauder bases}\label{sec-Schauder}

In this section, we provide the required background about closed convex cones generated by unconditional Schauder bases. Let $\{ e_k \}_{k \in \bbn}$ be an unconditional Schauder basis of the Hilbert space $H$; that is, for each $h \in H$ there is a unique sequence $(h_k)_{k \in \bbn} \subset \bbr$ such that
\begin{align}\label{series-h}
h = \sum_{k \in \bbn} h_k e_k,
\end{align}
and the series (\ref{series-h}) converges unconditionally.  Without loss of generality, we assume that $\| e_k \| = 1$ for all $k \in \bbn$.

\begin{remark}
Every orthonormal basis of the Hilbert space $H$ is an unconditional Schauder basis. Of course, the converse statement is not true, but for every unconditional Schauder basis of the Hilbert space $H$ there is an equivalent inner product on $H$ under which the unconditional Schauder basis is an orthonormal basis; see \cite{Bari}.
\end{remark}

There are unique elements $\{ e_k^* \}_{k \in \bbn} \subset H$ such that
\begin{align*}
\langle e_k^*,h \rangle = h_k \quad \text{for each $h \in H$,}
\end{align*}
where we refer to the series representation (\ref{series-h}); see \cite[page 164]{Fabian}. Given these coordinate functionals $\{ e_k^* \}_{k \in \bbn}$, we also call $\{ e_k^*,e_k \}_{k \in \bbn}$ an unconditional Schauder basis of $H$. Recall that, throughout this paper, we consider a closed convex cone $K \subset H$ with representation (\ref{cone-G}) for some generating system $G^* \subset K^*$. Now, we make an additional assumption on the generating system $G^*$ of the cone.

\begin{assumption}\label{ass-Schauder-basis}
We assume there is an unconditional Schauder basis $\{ e_k^*,e_k \}_{k \in \bbn}$ of $H$ such that
\begin{align*}
G^* \subset \{ \theta e_k^* : \theta \in \{ -1,1 \} \text{ and } k \in \bbn \}.
\end{align*}
\end{assumption}

\begin{remark}
Equivalently, we could demand $G^* \subset \bigcup_{k \in \bbn} \langle e_k^* \rangle$. Assumption \ref{ass-Schauder-basis} ensures that the generating system $G^*$ becomes minimal.
\end{remark}

We define the sequence $(E_n)_{n \in \bbn_0}$ of finite dimensional subspaces $E_n \subset H$ as $E_n := \langle e_1,\ldots,e_n \rangle$. Furthermore, we define the sequence $(\Pi_n)_{n \in \bbn_0}$ of projections $\Pi_n \in L(H,E_n)$ as
\begin{align}\label{Pi-def}
\Pi_n h = \sum_{k=1}^n \langle e_k^*,h \rangle e_k = \sum_{k=1}^n h_k e_k, \quad h \in H,
\end{align}
where we refer to the series representation (\ref{series-h}) of $h$. We denote by $\bc := \sup_{n \in \bbn} \| \Pi_n \|$ the basis constant of the Schauder basis $\{ e_k \}_{k \in \bbn}$. Since the Schauder basis is unconditional, by \cite[Prop. 6.31]{Fabian} there is a constant $C \in \bbr_+$ such for all $m \in \bbn$, all $\lambda_1,\ldots,\lambda_m \in \bbr$ and all $\epsilon_1,\ldots,\epsilon_m \in \{ -1,1 \}$ we have
\begin{align}\label{ubc-est}
\bigg\| \sum_{k=1}^m \epsilon_k \lambda_k e_k \bigg\| \leq C \bigg\| \sum_{k=1}^m \lambda_k e_k \bigg\|.
\end{align}
The smallest possible constant $C \in \bbr_+$ such that the inequality (\ref{ubc-est}) is fulfilled, is called the unconditional basis constant, and is denoted by $\ubc$.

\begin{lemma}\label{lemma-norm-g-star}
The following statements are true:
\begin{enumerate}
\item We have $1 \leq \bc \leq \ubc$.

\item For each $k \in \bbn$ we have $\| \langle e_k^*,\cdot \rangle \| \leq 2 \bc$.

\item For all $h \in H$ with representation (\ref{series-h}) and every bounded sequence $(\lambda_k)_{k \in \bbn}$ we have
\begin{align*}
g := \sum_{k \in \bbn} \lambda_k h_k e_k \in H
\end{align*}
with norm estimate
\begin{align*}
\| g \| \leq \ubc \Big( \sup_{k \in \bbn} |\lambda_k| \Big) \| h \|.
\end{align*}
\end{enumerate}
\end{lemma}

\begin{proof}
The first statement follows the proof of \cite[Prop. 6.31]{Fabian}. Noting that $\| e_k \| = 1$, by the Cauchy-Schwarz inequality, Assumption \ref{ass-Schauder-basis} and the identity
\begin{align*}
\| e_k^* \| \, \| e_k \| \leq 2 \bc
\end{align*}
from \cite[page 164]{Fabian}, for each $h \in H$ we obtain
\begin{align*}
|\langle e_k^*,h \rangle| \leq \| e_k^* \| \, \| h \| \leq  2 \bc \| h \|.
\end{align*}
The third statement follows from \cite[Lemma~6.33]{Fabian}.
\end{proof}

\begin{lemma}\label{lemma-proj-par}
The following statements are true:
\begin{enumerate}
\item We have $\Pi_n \to {\rm Id}_H$ as $n \to \infty$.

\item For all $k,n \in \bbn$, all $h^* \in \langle e_k^* \rangle$ and all $h \in H$ we have
\begin{align*}
\langle h^*,\Pi_n h \rangle = \langle h^*,h \rangle \mathbbm{1}_{\{ k \leq n \}}.
\end{align*}
\end{enumerate}
\end{lemma}

\begin{proof}
The first statement follows from \cite[Lemma 6.2.iii]{Fabian}, and the second statement follows from the definition (\ref{Pi-def}) of the projection $\Pi_n$.
\end{proof}

\section{Sufficiency of the invariance conditions for diffusion SPDEs with smooth volatilities}\label{sec-proof-1}

In this section, we prove the sufficiency of our invariance conditions for diffusion SPDEs (\ref{SPDE-Wiener}) with smooth volatilities. Recall that the distance function $d_K : H \to \bbr_+$ of the cone $K$ is given by
\begin{align*}
d_K(h) := \inf_{g \in K} \| h - g \|.
\end{align*}

\begin{lemma}\label{lemma-distance}
The following statements are true:
\begin{enumerate}
\item For all $\lambda \geq 0$ and $h \in H$ we have
\begin{align}\label{distance-1}
d_K(\lambda h) = \lambda d_K( h ).
\end{align}
\item For all $h \in H$ and $g \in K$ we have
\begin{align}
d_K(h+g) \leq d_K(h).
\end{align}
\end{enumerate}
\end{lemma}

\begin{proof}
Let $h \in H$ be arbitrary. For $\lambda = 0$ both sides in (\ref{distance-1}) are zero, and for $\lambda > 0$, by Definition \ref{def-cone} we obtain
\begin{align*}
d_K(\lambda h) = \inf_{g \in K} \| \lambda h - g \| = \inf_{f \in K} \| \lambda h - \lambda f \| = \lambda \inf_{f \in K} \| h - f \| = \lambda d_K( h ),
\end{align*}
proving the first statement. For the proof of the second statement, let $h \in H$ and $g \in K$ be arbitrary. Note that $K \subset K - \{ g \}$. Indeed, for each $f \in K$ by Definition \ref{def-conv-cone} we have $f + g \in K$, and hence $f = (f+g)-g \in K - \{ g \}$. This gives us
\begin{align*}
d_K(h+g) &= \inf_{f \in K} \| (h+g) - f \| = \inf_{f \in K} \| h - (f-g) \|
\\ &= \inf_{e \in K - \{ g \}} \| h - e \| \leq \inf_{e \in K} \| h - e \| = d_K(h),
\end{align*}
establishing the second statement.
\end{proof}

The following result ensures that the stochastic semigroup Nagumo's condition (SSNC) is fulfilled in our situation.

\begin{proposition}\label{prop-SSNC}
Let $\Sigma \in \F(H)$ be such that for all $(h^*,h) \in D$ we have
\begin{align}\label{drift-Nagumo}
\liminf_{t \downarrow 0} \frac{\langle h^*,S_t h \rangle}{t} + \langle h^*,\Sigma(h) \rangle \geq 0.
\end{align}
Then, for each $h \in K$ we have
\begin{align}\label{SSNC}
\liminf_{t \downarrow 0} \frac{1}{t} d_K(S_t h + t \Sigma(h)) = 0.
\end{align}
\end{proposition}

\begin{proof}
Since $\Sigma \in \F(H)$, there is an index $n \in \bbn$ such that $\Sigma(H) \subset E_n$. Let $h \in K$ be arbitrary.
We set $\bbn_n := \{ 1,\ldots,n \}$ and
\begin{align*}
\bbn_n^1 &:= \{ k \in \bbn_n : (e_k^*,h) \in D \text{ or } (-e_k^*,h) \in D \},
\\ \bbn_n^2 &:= \{ k \in \bbn_n : e_k^* \in G^* \text{ or } -e_k^* \in G^* \}
\\ &\qquad \cap \{ k \in \bbn_n : (e_k^*,h) \notin D \text{ and } (-e_k^*,h) \notin D \},
\\ \bbn_n^3 &:= \{ k \in \bbn_n : e_k^* \notin G^* \text{ and } -e_k^* \notin G^* \}.
\end{align*}
Then we have the decomposition $\bbn_n = \bbn_n^1 \cup \bbn_n^2 \cup \bbn_n^3$, for each $k \in \bbn_n^1$ there exists $\theta_k \in \{ -1,1 \}$ such that $(\theta_k e_k^*,h) \in D$, and for each $k \in \bbn_n^2$ there exists $\theta_k \in \{ -1,1 \}$ such that $\theta_k e_k^* \in G^*$ and $(\theta_k e_k^*,h) \notin D$. Furthermore, we set $\theta_k := 1$ for each $k \in \bbn_n^3$. There is a sequence $(t_m)_{m \in \bbn} \subset (0,\infty)$ with $t_m \downarrow 0$ such that
\begin{align}\label{SSNC-proof-1}
c_m(k) \geq 0 \quad \text{for all $m \in \bbn$ and all $k \in \bbn_n^2.$}
\end{align}
where we agree on the notation
\begin{align*}
c_m(k) := \frac{\langle \theta_k e_k^*,S_{t_m} h + t_m \Sigma(h) \rangle}{t_m} \quad \text{for all $m \in \bbn$ and all $k \in \bbn_n$.}
\end{align*}
Inductively, we define the subsequences $(m(k)_p)_{p \in \bbn}$ for $k \in \{ 0 \} \cup \bbn_n^1$ as follows:
\begin{enumerate}
\item For $k = 0$ we set $m(0)_p := p$ for each $p \in \bbn$.

\item Let $k \in \bbn_n^1$ be arbitrary, and suppose that we have defined $(m(l)_p)_{p \in \bbn}$, where $l$ denotes the largest integer from $\{ 0 \} \cup \bbn_n^1$ with $l < k$. We distinguish two cases:
\begin{itemize}
\item If $\liminf_{p \to \infty} c_{m(l)_p}(k) = \infty$, then we choose a subsequence $(m(k)_p)_{p \in \bbn}$ of $(m(l)_p)_{p \in \bbn}$ such that $c_{m(k)_p}(k) \geq 0$ for all $p \in \bbn$.

\item Otherwise, we choose a subsequence $(m(k)_p)_{p \in \bbn}$ of $(m(l)_p)_{p \in \bbn}$ such that $c_{m(k)_p}(k)$ converges to a finite limit for $p \to \infty$.
\end{itemize}
\end{enumerate}
Now, we define the subsequence $(m_p)_{p \in \bbn}$ as $m_p := m(k)_p$ for each $p \in \bbn$, where $k$ denotes the largest integer from $\{ 0 \} \cup \bbn_n^1$. Furthermore, we define the sets
\begin{align*}
\bbn_n^{1a} &:= \Big\{ k \in \bbn_n^1 : \liminf_{p \to \infty} c_{m_p}(k) < \infty \Big\},
\\ \bbn_n^{1b} &:= \Big\{ k \in \bbn_n^1 : \liminf_{p \to \infty} c_{m_p}(k) = \infty \Big\}.
\end{align*}
Then we have the decomposition $\bbn_n^1 = \bbn_n^{1a} \cup \bbn_n^{1b}$, and by (\ref{drift-Nagumo}) we have
\begin{align}\label{SSNC-proof-2}
\lim_{p \to \infty} c_{m_p}(k) &\in \bbr_+ \quad \text{for all $k \in \bbn_n^{1a},$}
\\ \label{SSNC-proof-3} c_{m_p}(k) &\geq 0 \quad \text{for all $p \in \bbn$ and all $k \in \bbn_n^{1b}.$}
\end{align}
Since $\Sigma(H) \subset E_n$, and $K$ is invariant for the semigroup $(S_t)_{t \geq 0}$ and $({\rm Id} - \Pi_n)$-invariant, by Lemma \ref{lemma-distance} and (\ref{SSNC-proof-1}), (\ref{SSNC-proof-3}), for each $p \in \bbn$ we obtain
\begin{align*}
&\frac{1}{t_{m_p}} d_K ( S_{t_{m_p}} h + t_{m_p} \Sigma(h) ) = \frac{1}{t_{m_p}} d_K \big( \underbrace{({\rm Id} - \Pi_n) S_{t_{m_p}} h}_{\in K} + \Pi_n(S_{t_{m_p}} h + t_{m_p} \Sigma(h)) \big) 
\\ &\leq \frac{1}{t_{m_p}} d_K \big( \Pi_n(S_{t_{m_p}} h + t_{m_p} \Sigma(h)) \big) = d_K \bigg( \Pi_n \frac{S_{t_{m_p}} h + t_{m_p} \Sigma(h)}{t_{m_p}} \bigg)
\\ &= d_K \bigg( \sum_{k \in \bbn_n^{1a}} c_{m_p}(k) \theta_k e_k + \underbrace{\sum_{k \in \bbn_n^{1b} \cup \bbn_n^2 \cup \bbn_n^3} c_{m_p}(k) \theta_k e_k}_{\in K} \bigg) \leq d_K \bigg( \sum_{k \in \bbn_n^{1a}} c_{m_p}(k) \theta_k e_k \bigg),
\end{align*}
and by the continuity of the distance function $d_K$ and (\ref{SSNC-proof-2}) we have
\begin{align*}
\lim_{p \to \infty} d_K \bigg( \sum_{k \in \bbn_n^{1a}} c_{m_p}(k) \theta_k e_k \bigg) = d_K \bigg( \underbrace{\sum_{k \in \bbn_n^{1a}} \lim_{p \to \infty} c_{m_p}(k) \theta_k e_k}_{\in K} \bigg) = 0,
\end{align*}
completing the proof.
\end{proof}

\begin{theorem}\label{thm-diffusion-C2}
Suppose that Assumptions \ref{ass-pseudo-contractive}, \ref{ass-cone-semigroup} and \ref{ass-Schauder-basis} are fulfilled, and that
\begin{align*}
\alpha &\in {\rm Lip}(H) \cap \F(H) \cap \B(H),
\\ \sigma &\in \G(H,L_2^0(H)) \cap \F(H,L_2^0(H)) \cap C_b^2(H,L_2^0(H)).
\end{align*}
If we have
\begin{align}\label{drift-cond-C2}
\liminf_{t \downarrow 0} \frac{\langle h^*,S_t h \rangle}{t} + \langle h^*,\alpha(h) \rangle \geq 0 \quad \text{for all $(h^*,h) \in D$,}
\end{align}
and for all $(h^*,h) \in D$ and each $j \in \bbn$ there exists $\epsilon = \epsilon(h^*,h,j) > 0$ such that
\begin{align}\label{loc-par-prop-result}
\langle h^*, \sigma^j(h-g) \rangle = 0 \quad \text{for all $g \in H$ with $\| g \| \leq \epsilon$,}
\end{align}
then the closed convex cone $K$ is invariant for the SPDE (\ref{SPDE-Wiener}).
\end{theorem}

\begin{proof}
Condition (\ref{loc-par-prop-result}) just means that for each $j \in \bbn$ the function $\sigma^j : H \to H$ is weakly locally parallel in the sense of Definition \ref{def-locally-par-weak}, which allows us to apply Lemma \ref{lemma-rho-parallel} in the sequel. Let $\rho : H \to H$ be the function defined in (\ref{def-rho}). According to our hypotheses and Lemma \ref{lemma-rho}, all assumptions from \cite{Nakayama} are satisfied. Let $u \in U_0$ be arbitrary, and define the function $\Sigma : H \to H$ as
\begin{align*}
\Sigma(h) := \alpha(h) - \rho(h) + \sigma(h)u, \quad h \in H.
\end{align*}
Since $\alpha \in \F(H)$ and $\sigma \in \F(H,L_2^0(H))$, we have $\Sigma \in \F(H)$. Let $(h^*,h) \in D$ be arbitrary. Then, by (\ref{drift-cond-C2}) and Lemmas \ref{lemma-rho-parallel}, \ref{lemma-sigma-u} we deduce that condition (\ref{drift-Nagumo}) is fulfilled. Therefore, by Proposition \ref{prop-SSNC} the SSNC (\ref{SSNC}) is fulfilled. Consequently, applying \cite[Prop. 1.1]{Nakayama} yields that the closed convex cone $K$ is invariant for the SPDE (\ref{SPDE-Wiener}).
\end{proof}

\section{Sufficiency of the invariance conditions for diffusion SPDEs with Lipschitz coefficients}\label{sec-proof-2}

In this section, we prove that our invariance conditions are sufficient for diffusion SPDEs (\ref{SPDE-Wiener}) with Lipschitz coefficients, without imposing smoothness on the volatility.

\begin{theorem}\label{thm-diffusion}
Suppose that Assumptions \ref{ass-pseudo-contractive}, \ref{ass-cone-semigroup} and \ref{ass-Schauder-basis} are fulfilled, and that $\alpha \in {\rm Lip}(H)$ and $\sigma \in {\rm Lip}(H,L_2^0(H))$. If for all $(h^*,h) \in D$ we have (\ref{drift-cond-C2}) and (\ref{main-4}), then the closed convex cone $K$ is invariant for the SPDE (\ref{SPDE-Wiener}).
\end{theorem}

\begin{proof}
For the proof of this result, we will apply the results from Appendices~\ref{app-drift} and \ref{app-volatility}. Note that Assumption \ref{ass-D} is fulfilled by virtue of Lemma \ref{lemma-fct-a}. Concerning the drift $\alpha$, we use the approximation results from Appendix~\ref{app-drift} as follows:
\begin{enumerate}
\item Condition (\ref{drift-cond-C2}) just means that $(a,\alpha)$ is inward pointing in the sense of Definition \ref{def-inward-pointing}.

\item By our stability result for SPDEs (Proposition \ref{prop-K-stability}) and Proposition \ref{prop-alpha-FDR} we may assume that
\begin{align*}
\alpha \in {\rm Lip}(H) \cap \F(H).
\end{align*}
\item By our stability result for SPDEs (Proposition \ref{prop-K-stability}) and Proposition \ref{prop-alpha-bounded} we may assume that
\begin{align*}
\alpha \in {\rm Lip}(H) \cap \F(H) \cap \B(H).
\end{align*}
\end{enumerate}
Furthermore, concerning the volatility $\sigma$, we use the approximation results from Appendix \ref{app-volatility} as follows:
\begin{enumerate}
\item Condition (\ref{main-4}) just means that for each $j \in \bbn$ the volatility $\sigma^j : H \to H$ is parallel in the sense of Definition \ref{def-parallel}.

\item By our stability result for SPDEs (Proposition \ref{prop-K-stability}) and Proposition \ref{prop-sigma-finite} we may assume that
\begin{align*}
\sigma \in {\rm Lip}(H,L_2^0(H)) \cap \G(H,L_2^0(H)).
\end{align*}
This allows us to apply the remaining results from Appendix \ref{app-volatility} (Propositions \ref{prop-sigma-FDR}--\ref{prop-C-b-2}), which are all stated for volatilities of the form $\sigma : H \to H$.

\item By our stability result for SPDEs (Proposition \ref{prop-K-stability}) and Proposition \ref{prop-sigma-FDR} we may assume that
\begin{align*}
\sigma \in {\rm Lip}(H,L_2^0(H)) \cap \F(H,L_2^0(H)).
\end{align*}

\item By our stability result for SPDEs (Proposition \ref{prop-K-stability}) and Proposition \ref{prop-sigma-bounded} we may assume that
\begin{align*}
\sigma \in {\rm Lip}(H,L_2^0(H)) \cap \F(H,L_2^0(H)) \cap \B(H,L_2^0(H)).
\end{align*}

\item By our stability result for SPDEs (Proposition \ref{prop-K-stability}) and Proposition \ref{prop-locally-par} we may assume that for each $j \in \bbn$ the volatility $\sigma^j : H \to H$ is locally parallel in the sense of Definition \ref{def-locally-par}.

\item By our stability result for SPDEs (Proposition \ref{prop-K-stability}) and Proposition \ref{prop-C-b-1-1} we may assume that
\begin{align*}
\sigma \in \F(H,L_2^0(H)) \cap C_b^{1,1}(H,L_2^0(H)),
\end{align*}
and that $\sigma^j : H \to H$ is locally parallel for each $j \in \bbn$.

\item By our stability result for SPDEs (Proposition \ref{prop-K-stability}) and Proposition \ref{prop-C-b-2} we may assume that
\begin{align*}
\sigma \in \F(H,L_2^0(H)) \cap C_b^2(H,L_2^0(H)),
\end{align*}
and that for each $j \in \bbn$ the volatility $\sigma^j : H \to H$ is weakly locally parallel in the sense of Definition \ref{def-locally-par-weak}.

\end{enumerate}
Consequently, applying Theorem \ref{thm-diffusion-C2} completes the proof.
\end{proof}

\section{Sufficiency of the invariance conditions for SPDEs with Lipschitz coefficients}\label{sec-proof-3}

In this section, we prove that our invariance conditions are sufficient for general jump-diffusion SPDEs (\ref{SPDE}) with Lipschitz coefficients.

\begin{theorem}\label{thm-suff}
Suppose that Assumptions \ref{ass-pseudo-contractive}, \ref{ass-cone-semigroup} and \ref{ass-Schauder-basis} are fulfilled, and that $\alpha \in \Lip(H)$, $\sigma \in \Lip(H,L_2^0(H))$ and $\gamma \in \Lip(H,L^2(F))$. If we have (\ref{main-1}), and for all $(h^*,h) \in D$ we have (\ref{main-3}) and (\ref{main-4}), then the closed convex cone $K$ is invariant for the SPDE (\ref{SPDE}).
\end{theorem}

\begin{proof}
Since the measure $F$ is $\sigma$-finite, by our stability result (Proposition \ref{prop-K-stability}) it suffices to prove that for each $B \in \cale$ with $F(B) < \infty$ the cone $K$ is invariant for the SPDE
\begin{align*}
\left\{
\begin{array}{rcl}
dr_t & = & ( A r_t + \alpha(r_t) - \int_B \gamma(r_t,x) F(dx) ) dt + \sigma(r_t) dW_t \medskip
\\ && + \int_B \gamma(r_{t-},x) \mu(dt,dx) \medskip
\\ r_0 & = & h_0.
\end{array}
\right.
\end{align*}
Moreover, by the jump condition (\ref{main-1}) and \cite[Lemmas 2.12 and 2.20]{FTT-appendix}, it suffices to prove that the cone $K$ is invariant for the SPDE
\begin{align}\label{SPDE-B-2}
\left\{
\begin{array}{rcl}
dr_t & = & ( A r_t + \alpha_B(r_t) ) dt + \sigma(r_t) dW_t \medskip
\\ r_0 & = & h_0.
\end{array}
\right.
\end{align}
where $\alpha_B : H \to H$ is given by
\begin{align*}
\alpha_B(h) := \alpha(h) - \int_B \gamma(h,x) F(dx), \quad h \in H.
\end{align*}
Note that by the Cauchy-Schwarz inequality we have $\alpha_B \in \Lip(H)$. Let $(h^*,h) \in D$ be arbitrary. By (\ref{main-3}) and Proposition \ref{prop-conditions} we obtain
\begin{align*}
&\liminf_{t \downarrow 0} \frac{\langle h^*,S_t h \rangle}{t} + \langle h^*, \alpha_B(h) \rangle = \liminf_{t \downarrow 0} \frac{\langle h^*,S_t h \rangle}{t} + \langle h^*,\alpha(h) \rangle
\\ &\quad - \int_E \langle h^*,\gamma(h,x) \rangle F(dx) + \int_{E \setminus B} \langle h^*,\gamma(h,x) \rangle F(dx) \geq 0.
\end{align*}
Therefore, applying Theorem \ref{thm-diffusion} yields that the cone $K$ is invariant for the SPDE (\ref{SPDE-B-2}), completing the proof.
\end{proof}

\section{Sufficiency of the invariance conditions and proof of the main result}\label{sec-proof-4}

In this section, we prove that our invariance conditions are sufficient for jump-diffusion SPDEs (\ref{SPDE-Wiener}) with coefficients being locally Lipschitz and satisfying the linear growth condition.

\begin{theorem}\label{thm-suff-general}
Suppose that Assumptions \ref{ass-pseudo-contractive}, \ref{ass-loc-Lip-LG}, \ref{ass-cone-semigroup} and \ref{ass-Schauder-basis} are fulfilled. If we have (\ref{main-1}), and for all $(h^*,h) \in D$ we have (\ref{main-3}) and (\ref{main-4}), then the closed convex cone $K$ is invariant for the SPDE (\ref{SPDE}).
\end{theorem}

\begin{proof}
Let $h_0 \in K$ be arbitrary. Let $(R_n)_{n \in \bbn}$ be the sequence of retractions $R_n : H \to H$ defined according to Definition \ref{def-retract-X}. We define the sequences of functions $(\alpha_n)_{n \in \bbn}$, $(\sigma_n)_{n \in \bbn}$ and $(\gamma_n)_{n \in \bbn}$ as
\begin{align*}
\alpha_n := \alpha \circ R_n, \quad \sigma_n := \sigma \circ R_n \quad \text{and} \quad \gamma_n := \gamma \circ R_n.
\end{align*}
Let $n \in \bbn$ be arbitrary. Then, by Lemma \ref{lemma-retract-X} we have
\begin{align*}
\alpha_n \in \Lip(H), \quad \sigma_n \in \Lip(H,L_2^0(H)) \quad \text{and} \quad \gamma \in \Lip(H,L^2(F)),
\end{align*}
and hence, there exists a unique mild solution $r^n$ to the SPDE (\ref{SPDE-n}) with $r_0^n = h_0$. Now, we check that conditions (\ref{main-1})--(\ref{main-4}) are fulfilled with $(\alpha,\sigma,\gamma)$ replaced by $(\alpha_n,\sigma_n,\gamma_n)$. Following the notation from Definition \ref{def-retract-X}, there is a function $\lambda_n : H \to (0,1]$ such that
\begin{align*}
R_n(h) = \lambda_n(h) h \quad \text{for all $h \in H$.}
\end{align*}
Let $h \in K$ be arbitrary. By the properties of the closed convex cone $K$ we have $\lambda_n(h) h \in K$ and $(1-\lambda_n(h)) h \in K$, and hence, since condition (\ref{main-1}) is satisfied for $\gamma$, we obtain
\begin{align*}
h + \gamma_n(h,x) = h + \gamma(\lambda_n(h) h,x) = \underbrace{(1 - \lambda_n(h)) h}_{\in K} + \underbrace{\lambda_n(h) h + \gamma(\lambda_n(h) h,x)}_{\in K} \in K
\end{align*}
for $F$-almost all $x \in E$, showing (\ref{main-1}) with $\gamma$ replaced by $\gamma_n$. Now, let $h^* \in G^*$ be such that $(h^*,h) \in D$. Then, by Lemma \ref{lemma-fct-a} we also have $(h^*,\lambda_n(h) h) \in D$, and since condition (\ref{main-4}) is satisfied for $\sigma$, we obtain
\begin{align*}
\langle h^*,\sigma_n^j(h) \rangle = \langle h^*,\sigma^j(\lambda_n(h)h) \rangle = 0, \quad j \in \bbn,
\end{align*}
showing (\ref{main-4}) with $\sigma$ replaced by $\sigma_n$. Furthermore, since condition (\ref{main-3}) is satisfied for $(\alpha,\gamma)$, we obtain
\begin{align*}
&\liminf_{t \downarrow 0} \frac{\langle h^*,S_t h \rangle}{t} + \langle h^*,\alpha_n(h) \rangle - \int_E \langle h^*,\gamma_n(h,x) \rangle F(dx)
\\ &= \liminf_{t \downarrow 0} \frac{\langle h^*,S_t h \rangle}{t} + \langle h^*,\alpha(\lambda_n(h)h) \rangle - \int_E \langle h^*,\gamma(\lambda_n(h)h,x) \rangle F(dx)
\\ &\geq (1-\lambda_n(h)) \liminf_{t \downarrow 0} \frac{\langle h^*,S_t h \rangle}{t} + \liminf_{t \downarrow 0} \frac{\langle h^*,S_t (\lambda_n(h)h) \rangle}{t}
\\ &\quad + \langle h^*,\alpha(\lambda_n(h)h) \rangle - \int_E \langle h^*,\gamma(\lambda_n(h)h,x) \rangle F(dx) \geq 0,
\end{align*}
showing (\ref{main-3}) with $(\alpha,\gamma)$ replaced by $(\alpha_n,\gamma_n)$. Consequently, by Theorem \ref{thm-suff} we have $r^n \in K$ up to an evanescent set. Now, we define the increasing sequence $(T_n)_{n \in \bbn_0}$ of stopping times by $T_0 := 0$ and
\begin{align*}
T_n := \inf \{ t \in \bbr_+ : \| r_t^n \| > n \} \quad \text{for all $n \in \bbn$.}
\end{align*}
Then we have $\bbp(T_n \to \infty) = 1$, and the mild solution $r$ to (\ref{SPDE}) with $r_0 = h_0$ is given by
\begin{align}\label{r-lin-growth}
r = h_0 \mathbbm{1}_{[\![ T_0 ]\!]} + \sum_{n \in \mathbb{N}} r^n \mathbbm{1}_{]\!] T_{n-1}, T_n ]\!]},
\end{align}
showing that $r \in K$ up to an evanescent set.
\end{proof}

Now, we are ready to provide the proof of our main result, which concludes the paper.

\begin{proof}[Proof of Theorem \ref{thm-main}]
(i) $\Rightarrow$ (ii): This implication follows from Theorem \ref{thm-nec}.

\noindent (ii) $\Rightarrow$ (i): This implication follows from Theorem \ref{thm-suff-general}.
\end{proof}

\section{An example}\label{sec-example}

In this section, we provide an example illustrating our main result. Let $H = \ell^2(\bbn)$ be the Hilbert space consisting of all sequences $h = (h_k)_{k \in \bbn} \subset \bbr$ such that $\sum_{k \in \bbn} |h_k|^2 < \infty$. As in \cite[Example 2.5.4]{Pazy}, let $(S_t)_{t \geq 0}$ be the semigroup given by
\begin{align}\label{def-semigroup}
S_t h := (e^{-kt} h_k)_{k \in \bbn} \quad \text{for $t \geq 0$ and $h = (h_k)_{k \in \bbn} \in H$.}
\end{align}
Then $(S_t)_{t \geq 0}$ is a $C_0$-semigroup with infinitesimal generator $A : \cald(A) \subset H \to H$ defined on the domain
\begin{align*}
\cald(A) = \{ (h_k)_{k \in \bbn} \in H : (k h_k)_{k \in \bbn} \in H \},
\end{align*}
and given by
\begin{align*}
Ah = (-k h_k)_{k \in \bbn} \quad \text{for $h = (h_k)_{k \in \bbn} \in \cald(A)$.}
\end{align*}
We consider the closed convex cone
\begin{align*}
K := \{ h = (h_k)_{k \in \bbn} \in H : h_k \geq 0 \text{ for all } k \in \bbn \}
\end{align*}
consisting of all nonnegative sequences.

\begin{proposition}
Suppose that Assumption \ref{ass-loc-Lip-LG} is fulfilled.
Then the following statements are equivalent:
\begin{enumerate}
\item[(i)] The closed convex cone $K$ is invariant for the SPDE (\ref{SPDE}).

\item[(ii)] We have
\begin{align*}
h + \gamma(h,x) \in K \quad \text{for $F$-almost all $x \in E$,} \quad \text{for all $h \in K$,}
\end{align*}
and for all $(k,h) \in \bbn \times K$ with $h_k = 0$ we have
\begin{align*}
& \alpha_k(h) - \int_E \gamma_k(h,x) F(dx) \geq 0,
\\ &\sigma_k^j(h) = 0, \quad j \in \bbn.
\end{align*}
\end{enumerate}
\end{proposition}

\begin{proof}
By definition (\ref{def-semigroup}) the semigroup $(S_t)_{t \geq 0}$ is a semigroup of contractions, and the cone $K$ is invariant for the semigroup $(S_t)_{t \geq 0}$, showing that Assumptions \ref{ass-pseudo-contractive} and \ref{ass-cone-semigroup} are fulfilled. Moreover, the cone $K$ is self-dual; that is $K^* = K$, and we have the representation
\begin{align*}
K = \bigcap_{h^* \in G^*} \{ h \in H : \langle h^*,h \rangle \geq 0 \},
\end{align*}
where $G^* \subset K^*$ is given by $G^* = \{ e_k : k \in \bbn \}$, showing that Assumption \ref{ass-Schauder-basis} is satisfied. Furthermore, for all $(k,h) \in \bbn \times K$ we have
\begin{align*}
\liminf_{t \downarrow 0} \frac{e^{-kt} h_k}{t} < \infty \quad \text{if and only if} \quad h_k = 0,
\end{align*}
and in this case the limes inferior vanishes. Consequently, applying Theorem \ref{thm-main} completes the proof.
\end{proof}

\begin{appendix}

\section{Function spaces}\label{app-function-spaces}

In this appendix, we collect the function spaces used in this paper. Let $X$ and $Y$ be two normed spaces.

\begin{definition}
We introduce the following notions:
\begin{enumerate}
\item For a constant $L \in \bbr_+$ a function $f : X \to Y$ is called \emph{$L$-Lipschitz} if
\begin{align}\label{Lip-cond-A}
\| f(x) - f(y) \| \leq L \| x-y \| \quad \text{for all $x,y \in X$.}
\end{align}
\item For a constant $L \in \bbr_+$ we define the space
\begin{align*}
{\rm Lip}_L(X,Y) := \{ f : X \to Y : \text{$f$ is $L$-Lipschitz} \}.
\end{align*}
\item A function $f \in {\rm Lip}_L(X,Y)$ is called \emph{Lipschitz continuous}.

\item We define the space ${\rm Lip}(X,Y) := \bigcup_{L \in \bbr_+} {\rm Lip}_L(X,Y)$.

\item For a constant $L \in \bbr_+$ we define the space ${\rm Lip}_L(X) := {\rm Lip}_L(X,X)$.

\item We define the space ${\rm Lip}(X) := {\rm Lip}(X,X)$.
\end{enumerate}
\end{definition}

\begin{definition}
We introduce the following notions:
\begin{enumerate}
\item A function $f : X \to Y$ is called \emph{locally Lipschitz} if for each $C \in \bbr_+$ there is a constant $L(C) \in \bbr_+$ such that
\begin{align*}
\| f(x) - f(y) \| \leq L(C) \| x-y \| \quad \text{for all $x,y \in X$ with $\| x \|, \| y \| \leq C$.}
\end{align*}
\item We denote by $\Lip^{\loc}(X,Y)$ the space of all locally Lipschitz functions $f : X \to Y$.

\item We define the space $\Lip^{\rm loc}(X) := \Lip^{\loc}(X,X)$.
\end{enumerate}
\end{definition}

\begin{definition}
We introduce the following notions:
\begin{enumerate}
\item We say that a function $f : X \to Y$ satisfies the \emph{linear growth condition} if there is a finite constant $C \in \bbr_+$ such that
\begin{align}\label{LG-A}
\| f(x) \| \leq C (1 + \| x \|) \quad \text{for all $x \in X$.}
\end{align}
\item We denote by $\LG(X,Y)$ the space of all functions $f : X \to Y$ satisfying the linear growth condition.

\item We define the space $\LG(X) := \LG(X,Y)$.
\end{enumerate}
\end{definition}

Note that $\Lip(X,Y) \subset \Lip^{\rm loc}(X,Y) \cap \LG(X,Y)$. Indeed, if (\ref{Lip-cond-A}) is fulfilled, setting $C := \max \{ L, \| f(0) \| \}$, for all $x \in X$ we obtain
\begin{align*}
\| f(x) \| \leq \| f(x) - f(0) \| + \| f(0) \| \leq L \| x \| + \| f(0) \| \leq C \| x \| + C = C (1 + \| x \|),
\end{align*}
showing (\ref{LG-A}).

\begin{definition}
We introduce the following notions:
\begin{enumerate}
\item A function $f : X \to Y$ is called \emph{bounded} if there is a constant $M \in \bbr_+$ such that
\begin{align*}
\| f(x) \| \leq M \quad \text{for all $x \in X$.}
\end{align*}
\item We denote by $\B(X,Y)$ the space of all bounded functions $f : X \to Y$.

\item We define the space $\B(X) := \B(X,X)$.
\end{enumerate}
\end{definition}

\begin{definition}
We introduce the following notions:
\begin{enumerate}
\item A function $f : X \to Y$ is called \emph{locally bounded} if for each $C \in \bbr_+$ there is a constant $M(C) \in \bbr_+$ such that
\begin{align}\label{loc-bounded-A}
\| f(x) \| \leq M(C) \quad \text{for all $x \in X$ with $\| x \| \leq C$.}
\end{align}
\item We denote by $\B^{\rm loc}(X,Y)$ the space of all locally bounded functions $f : X \to Y$.

\item We define the space $\B^{\rm loc}(X) := \B^{\rm loc}(X,X)$.
\end{enumerate}
\end{definition}

Note that $\LG(X,Y) \subset \B^{\loc}(X,Y)$. Indeed, if (\ref{LG-A}) is satisfied, for each $C \in \bbr_+$ we set $M(C) := C(1+C)$, and then for all $x \in X$ with $\| x \| \leq C$ we obtain
\begin{align*}
\| f(x) \| \leq C(1 + \| x \|) \leq C(1 + C) = M(C),
\end{align*}
showing (\ref{loc-bounded-A}).

\begin{definition}
We introduce the following notions:
\begin{enumerate}
\item We denote by $C(X,Y)$ the space of all continuous functions $f : X \to Y$.

\item We define the space $C_b(X,Y) := C(X,Y) \cap \B(X,Y)$.

\item We define the spaces $C(X) := C(X,X)$ and $C_b(X) := C_b(X,X)$.
\end{enumerate}
\end{definition}

Note that $\Lip^{\loc}(X,Y) \subset C(X,Y)$. For the next definition, we agree about the convention $\overline{\bbn} := \bbn \cup \{ \infty \}$, where $\bbn = \{ 1,2,3,\ldots \}$ denotes the natural numbers.

\begin{definition}
Let $p \in \overline{\bbn}$ be arbitrary.
\begin{enumerate}
\item We denote by $C^p(X,Y)$ the space of all $p$-times continuously differentiable functions $f : X \to Y$.

\item We denote by $C_b^p(X,Y)$ the space of all $f \in C^p(X,Y)$ such that $f$ is bounded and the derivatives $D^k f$, $k=1,\ldots,p$ are bounded.

\item We define the spaces $C^p(X) := C^p(X,X)$ and $C_b^p(X) := C_b^p(X,X)$. 
\end{enumerate}
\end{definition}

Note that $C_b^1(X,Y) \subset {\rm Lip}(X,Y) \cap \B(X,Y)$.

\begin{definition}
We introduce the following notions:
\begin{enumerate}
\item We denote by $C_b^{1,1}(X,Y)$ the space of all $f \in C_b^1(X,Y)$ such that $Df \in {\rm Lip}(X,L(X,Y))$.

\item We define the space $C_b^{1,1}(X) := C_b^{1,1}(X,X)$.
\end{enumerate}
\end{definition}

Note that $C_b^2(X,Y) \subset C_b^{1,1}(X,Y) \subset C_b^1(X,Y)$. For the rest of this section, let $H$ be a Hilbert space.

\begin{definition}\label{def-retract-X}
For each $n \in \bbn$ we define the retraction
\begin{align*}
R_n : H \to H, \quad R_n(h) := \lambda_n(h) h,
\end{align*}
where the function $\lambda_n : H \to (0,1]$ is given by
\begin{align*}
\lambda_n(h) := \bbI_{\{ \| h \| \leq n \}} + \frac{n}{\| h \|} \bbI_{\{ \| h \| > n \}}, \quad h \in H.
\end{align*}
\end{definition}

\begin{lemma}\label{lemma-retract-X}
The following statements are true:
\begin{enumerate}
\item We have $R_n \to {\rm Id}_H$ as $n \to \infty$.

\item For each $n \in \bbn$ we have $R_n \in \Lip_1(H) \cap \B(H)$.
\end{enumerate}
\end{lemma}

\begin{proof}
The first statement directly follows from Definition \ref{def-retract-X}. For the proof of the second statement, let $n \in \bbn$ be arbitrary. Then we have $\| R_n(h) \| \leq n$ for all $h \in H$, and hence $R_n \in \B(H)$. Furthermore, the ball $K_n := \{ h \in H : \| h \| \leq n \}$ is a closed convex set. Let $h \in H$ and $g \in K_n$ be arbitrary. If $h \in K_n$, then we have $R_n(h) = h$, and hence
\begin{align*}
\langle h - R_n(h),g - R_n(h) \rangle = 0.
\end{align*}
Now, suppose that $h \in H \setminus K_n$. By the Cauchy Schwarz inequality we have
\begin{align*}
\langle h,g \rangle \leq |\langle h,g \rangle| \leq \| h \| \, \| g \| \leq n \| h \|.
\end{align*}
Moreover, we have $\lambda_n(h) \| h \| = n$, and it follows that
\begin{align*}
\langle h - R_n(h),g - R_n(h) \rangle &= \langle h - \lambda_n(h) h, g - \lambda_n(h) h \rangle = \langle (1 - \lambda_n(h)) h, g - \lambda_n(h) h \rangle
\\ &= (1 - \lambda_n(h)) \big( \langle h,g \rangle - \lambda_n(h) \| h \|^2 \big)
\\ &\leq (1 - \lambda_n(h)) \big( n \| h \| - \lambda_n(h) \| h \|^2 \big)
\\ &= (1 - \lambda_n(h)) \| h \| ( n - \lambda_n(h) \| h \| ) = 0.
\end{align*}
Consequently, the mapping $R_n$ is the metric projection onto the closed convex set $K_n$, and therefore we have $R_n \in \Lip_1(H)$.
\end{proof}

\section{Stability result for SPDEs}\label{app-stability}

In this appendix, we present the required stability result for SPDEs. The mathematical framework is that of Section \ref{sec-ass}. Apart from the SPDE (\ref{SPDE}), we consider the sequence of SPDEs given by
\begin{align}\label{SPDE-n}
\left\{
\begin{array}{rcl}
dr_t^n & = & (A r_t^n + \alpha_n(r_t^n)) dt + \sigma_n(r_t^n) dW_t + \int_E \gamma_n(r_{t-}^n,x) (\mu(dt,dx) - F(dx) dt) \medskip
\\ r_0^n & = & h_0
\end{array}
\right.
\end{align}
for each $n \in \bbn$.

\begin{assumption}\label{ass-stability}
We suppose that the following conditions are fulfilled:
\begin{enumerate}
\item There exists $L \in \bbr_+$ such that $\alpha_n \in {\rm Lip}_L(H)$, $\sigma_n \in {\rm Lip}_L(H,L_2^0(H))$ and $\gamma_n \in {\rm Lip}_L(H,L^2(F))$ for all $n \in \bbn$.

\item We have $\alpha_n \to \alpha$, $\sigma_n \to \sigma$ and $\gamma_n \to \gamma$ for $n \to \infty$.
\end{enumerate}
\end{assumption}

\begin{proposition}\label{prop-stability}
Suppose that Assumption \ref{ass-stability} is fulfilled. Then, for each $h_0 \in H$ we have
\begin{align*}
\bbe \bigg[ \sup_{t \in [0,T]} \| r_t - r_t^n \|^2 \bigg] \to 0 \quad \text{for every $T \in \bbr_+$,}
\end{align*}
where $r$ denotes the mild solution to (\ref{SPDE}) with $r_0 = h_0$, and for each $n \in \bbn$ the process $r^n$ denotes the mild solution to (\ref{SPDE-n}) with $r_0^n = h_0$.
\end{proposition}

\begin{proof}
This is a consequence of \cite[Prop. 9.1.2]{SPDE}.
\end{proof}

\begin{proposition}\label{prop-K-stability}
Suppose that Assumption \ref{ass-stability} is fulfilled, and that for each $n \in \bbn$ the closed convex cone $K$ is invariant for the SPDE (\ref{SPDE-n}). Then $K$ is also invariant for the SPDE (\ref{SPDE}).
\end{proposition}

\begin{proof}
Let $h_0 \in K$ be arbitrary. We denote by $r$ the mild solution to (\ref{SPDE}) with $r_0 = h_0$, and for each $n \in \bbn$ we denote by $r^n$ the mild solution to (\ref{SPDE-n}) with $r_0^n = h_0$. Then, for each $n \in \bbn$ there is an event $\tilde{\Omega}_n \in \calf$ with $\bbp(\tilde{\Omega}_n) = 1$ such that $r_t^n(\omega) \in K$ for all $(\omega,t) \in \tilde{\Omega}_n \times \bbr_+$. Setting $\tilde{\Omega} := \bigcap_{n \in \bbn} \tilde{\Omega}_n \in \calf$ we have $\bbp(\tilde{\Omega}) = 1$ and $r_t^n(\omega) \in K$ for all $(\omega,t) \in \tilde{\Omega} \times \bbr_+$ and all $n \in \bbn$. Now, let $N \in \bbn$ be arbitrary. By Proposition \ref{prop-stability} we have
\begin{align*}
\bbe \bigg[ \sup_{t \in [0,N]} \| r_t - r_t^n \|^2 \bigg] \to 0,
\end{align*}
and hence, there is a subsequence $(n_k)_{k \in \bbn}$ such that $\bbp$-almost surely
\begin{align*}
\sup_{t \in [0,N]} \| r_t - r_t^{n_k} \| \to 0.
\end{align*}
Since $K$ is closed, there is an event $\bar{\Omega}_N \in \calf$ with $\bbp(\bar{\Omega}_N) = 1$ such that $r_t(\omega) \in K$ for all $(\omega,t) \in \bar{\Omega}_N \times [0,N]$. Therefore, setting $\bar{\Omega} := \bigcap_{N \in \bbn} \bar{\Omega}_N \in \calf$ we obtain $\bbp(\bar{\Omega}) = 1$ and $r_t(\omega) \in K$ for all $(\omega,t) \in \bar{\Omega} \times \bbr_+$, showing that $K$ is invariant for (\ref{SPDE}).
\end{proof}

\section{Inward pointing functions}\label{app-drift}

In this appendix, we provide the required results about inward pointing functions, which we need for the proof of Theorem \ref{thm-diffusion}. As in Section \ref{sec-ass}, let $H$ be a separable Hilbert space, let $K \subset H$ be a closed convex cone, and let $G^* \subset K^*$ be a generating system of the cone such that Assumption \ref{ass-Schauder-basis} is fulfilled. Let $D \subset G^* \times K$ be a subset, and let $a : D \to \bbr_+$ be a function.

\begin{assumption}\label{ass-D}
We suppose that for each $(h^*,h) \in D$ the following conditions are fulfilled:
\begin{enumerate}
\item We have $\langle h^*,h \rangle = 0$.

\item For all $\lambda \geq 0$ we have $(h^*,\lambda h) \in D$ and
\begin{align*}
a(h^*,\lambda h) = \lambda a(h^*,h).
\end{align*}
\item For all $g \in K$ with $g \leq_K h$ we have $(h^*,g) \in D$ and
\begin{align*}
a(h^*,g) \leq a(h^*,h).
\end{align*}
\end{enumerate}
\end{assumption}

\begin{definition}\label{def-inward-pointing}
Let $\alpha : H \to H$ be a function. We call the pair $(a,\alpha)$ \emph{inward pointing at the boundary of $K$} (in short \emph{inward pointing}) if for all $(h^*,h) \in D$ we have
\begin{align*}
a(h^*,h) + \langle h^*, \alpha(h) \rangle \geq 0.
\end{align*}
\end{definition}

\begin{definition}\label{def-parallel}
A function $\sigma : H \to H$ is called \emph{parallel at the boundary of $K$} (in short \emph{parallel}) if for all $(h^*,h) \in D$ we have
\begin{align*}
\langle h^*, \sigma(h) \rangle = 0.
\end{align*}
\end{definition}

\begin{definition}\label{def-inv-set}
Let $\sigma : H \to H$ be a function. Then the set $D$ is called \emph{$({\rm Id}_H,\sigma)$-invariant} if
\begin{align*}
(h^*,\sigma(h)) \in D \quad \text{for all $(h^*,h) \in D$.}
\end{align*}
\end{definition}

\begin{remark}
Let $\sigma : H \to H$ be a function. If $D$ is $({\rm Id}_H,\sigma)$-invariant, then $\sigma$ is parallel.
\end{remark}

\begin{lemma}\label{lemma-Pi-alpha}
Let $\alpha : H \to H$ be a function such that $(a,\alpha)$ is inward pointing. Then, for each $n \in \bbn$ the pair $(a,\Pi_n \circ \alpha)$ is inward pointing, too.
\end{lemma}

\begin{proof}
Let $(h^*,h) \in D$ be arbitrary. By Assumption \ref{ass-Schauder-basis} we have $h^* \in \langle e_k^* \rangle$ for some $k \in \bbn$. Thus, by Lemma \ref{lemma-proj-par}, and since $a$ is nonnegative, we obtain
\begin{align*}
a(h^*,h) + \langle h^*,\Pi_n(\alpha(h)) \rangle = a(h^*,h) + \langle h^*,\alpha(h) \rangle \mathbbm{1}_{\{ k \leq n \}} \geq 0,
\end{align*}
finishing the proof.
\end{proof}

\begin{definition}\label{def-F-spaces}
We introduce the following spaces:
\begin{enumerate}
\item For each $n \in \bbn$ we denote by $\F_n(H)$ the space of all functions $\alpha : H \to E_n$.

\item We set $\F(H) := \bigcup_{n \in \bbn} \F_n(H)$.
\end{enumerate}
\end{definition}

\begin{proposition}\label{prop-alpha-FDR}
Let $\alpha \in \Lip(H)$ be a function such that $(a,\alpha)$ is inward pointing. Then, there are a constant $L \in \bbr_+$ and a sequence
\begin{align}\label{alpha-n-FDR}
(\alpha_n)_{n \in \bbn} \subset \Lip_L(H) \cap \F(H) 
\end{align}
such that $(a,\alpha_n)$ is inward pointing for each $n \in \bbn$, and we have $\alpha_n \to \alpha$.
\end{proposition}

\begin{proof}
We set $\alpha_n := \Pi_n \circ \alpha$ for each $n \in \bbn$. Then, by construction for each $n \in \bbn$ we have $\alpha_n \in \F(H)$. By hypothesis there exists a constant $M \in \bbr_+$ such that $\alpha \in {\rm Lip}_M(H)$. Setting $L := M \bc$, we have $\alpha_n \in {\rm Lip}_L(H)$ for each $n \in \bbn$, showing (\ref{alpha-n-FDR}). Furthermore, by Lemma \ref{lemma-Pi-alpha}, for each $n \in \bbn$ the pair $(a,\alpha_n)$ is inward pointing, and by Lemma \ref{lemma-proj-par} we have $\alpha_n \to \alpha$.
\end{proof}

\begin{lemma}\label{lemma-inward-comp}
Let $\alpha,\beta : H \to H$ be two functions such that the following conditions are fulfilled:
\begin{enumerate}
\item $(a,\alpha)$ is inward pointing.

\item $D$ is $({\rm Id}_H,\beta)$-invariant, and for all $(h^*,h) \in D$ we have
\begin{align}\label{inward-comp-drift}
a(h^*,\beta(h)) \leq a(h^*,h).
\end{align}
\end{enumerate}
Then the pair $(a,\alpha \circ \beta)$ is inward pointing.
\end{lemma}

\begin{proof}
Let $(h^*,h) \in D$ be arbitrary. Since the set $D$ is $({\rm Id}_H,\beta)$-invariant, we have $(h^*,\beta(h)) \in D$. Therefore, by (\ref{inward-comp-drift}), and since $(a,\alpha)$ is inward pointing, we obtain
\begin{align*}
a(h^*,h) + \langle h^*,\alpha(\beta(h)) \rangle \geq a(h^*,\beta(h)) + \langle h^*,\alpha(\beta(h)) \rangle \geq 0,
\end{align*}
finishing the proof.
\end{proof}

We denote $(R_n)_{n \in \bbn}$ the retractions $R_n : H \to H$ defined according to Definition \ref{def-retract-X}.
We will need the following auxiliary result.

\begin{lemma}\label{lemma-retract-par}
Let $n \in \bbn$ be arbitrary. Then $D$ is $({\rm Id}_H,R_n)$-invariant, and for all $(h^*,h) \in D$ we have
\begin{align*}
a(h^*,R_n(h)) \leq a(h^*,h).
\end{align*}
\end{lemma}

\begin{proof}
Let $n \in \bbn$ be arbitrary. Recalling the notation from Definition \ref{def-retract-X}, there is a function $\lambda_n : H \to (0,1]$ such that
\begin{align*}
R_n(h) = \lambda_n(h) h \quad \text{for each $h \in H$.}
\end{align*}
By Assumption \ref{ass-D}, for all $(h^*,h) \in D$ we obtain $(h^*,R_n(h)) = (h^*,\lambda_n(h) h) \in D$ and
\begin{align*}
a(h^*,R_n(h)) = a(h^*,\lambda_n(h) h) = \lambda_n(h) a(h^*,h) \leq a(h^*,h),
\end{align*}
completing the proof.
\end{proof}

\begin{proposition}\label{prop-alpha-bounded}
Let $\alpha \in {\rm Lip}(H) \cap \F(H)$ be a function such that $(a,\alpha)$ is inward pointing. Then there are a constant $L \in \bbr_+$ and a sequence 
\begin{align}\label{alpha-bounded}
(\alpha_n)_{n \in \bbn} \subset {\rm Lip}_L(H) \cap \F(H) \cap \B(H)
\end{align}
such that $(a,\alpha_n)$ is inward pointing for each $n \in \bbn$, and we have $\alpha_n \to \alpha$.
\end{proposition}

\begin{proof}
We set $\alpha_n := \alpha \circ R_n$ for each $n \in \bbn$. Let $n \in \bbn$ be arbitrary. Then we have $\alpha_n \in \F(H)$, because $\alpha \in \F(H)$. By hypothesis there exists a constant $L \in \bbr_+$ such that $\alpha \in {\rm Lip}_L(H)$, and by Lemma \ref{lemma-retract-X} and the inclusion $\Lip_L(H) \subset \B^{\loc}(H)$ it follows that $\alpha_n \in {\rm Lip}_L(H) \cap \B(H)$, showing (\ref{alpha-bounded}). Combining Lemmas \ref{lemma-inward-comp} and \ref{lemma-retract-par}, we obtain that $(a,\alpha_n)$ is inward pointing. Furthermore, by Lemma \ref{lemma-retract-X} we have $\alpha_n \to \alpha$.
\end{proof}

\section{Parallel functions}\label{app-volatility}

In this appendix, we provide the required results about parallel function, which we need for the proofs of Theorems \ref{thm-diffusion-C2} and \ref{thm-diffusion}. The general mathematical framework is that of Appendix \ref{app-drift}. First, we will extend the Definition \ref{def-parallel} of a parallel function. 

\begin{definition}\label{def-locally-par}
A function $\sigma : H \to H$ is called \emph{locally parallel to the boundary of $K$} (in short \emph{locally parallel}) if there exists $\epsilon > 0$ such that for all $(h^*,h) \in D$ we have
\begin{align}\label{loc-par-prop}
\langle h^*, \sigma(h-g) \rangle = 0 \quad \text{for all $g \in H$ with $\| g \| \leq \epsilon$.}
\end{align}
\end{definition}

\begin{definition}\label{def-locally-par-weak}
A function $\sigma : H \to H$ is called \emph{weakly locally parallel to the boundary of $K$} (in short \emph{weakly locally parallel}) if for all $(h^*,h) \in D$ there exists $\epsilon = \epsilon(h^*,h) > 0$ such that we have (\ref{loc-par-prop}). 
\end{definition}

\begin{definition}
Let $\sigma : H \to H$ be a function. Then the set $D$ is called \emph{locally $({\rm Id}_H,\sigma)$-invariant} if there exists $\epsilon > 0$ such that for all $(h^*,h) \in D$ we have
\begin{align*}
(h^*,\sigma(h-g)) \in D \quad \text{for all $g \in H$ with $\| g \| \leq \epsilon$.}
\end{align*}
\end{definition}

\begin{remark}
Let $\sigma : H \to H$ be a function. 
\begin{enumerate}
\item If $\sigma$ is locally parallel, then it weakly locally parallel, too.

\item If $D$ is locally $({\rm Id}_H,\sigma)$-invariant, then $\sigma$ is locally parallel.
\end{enumerate}
\end{remark}

As in Section \ref{sec-ass}, let $U$ be a separable Hilbert space, and let $Q \in L(U)$ be a nuclear, self-adjoint, positive definite linear operator. Recall that $U_0 := Q^{1/2}(U)$ equipped with the inner product (\ref{inner-prod-U0}) is another separable Hilbert space, and that $L_2^0(H) := L_2(U_0,H)$ denotes the space of Hilbert-Schmidt operators from $U_0$ into $H$. Furthermore, recall that we have fixed an orthonormal basis $\{ g_j \}_{j \in \bbn}$ of $U_0$, and that for each $\sigma \in L_2^0(H)$ we set $\sigma^j := \sigma g_j$ for $j \in \bbn$. With this notation, the Hilbert-Schmidt norm is given by
\begin{align}\label{norm-HS}
\| \sigma \| = \sqrt{\sum_{j \in \bbn} \| \sigma^j \|^2} \quad \text{for each $\sigma \in L_2^0(H)$.}
\end{align}

\begin{definition}
We denote by $\F(H,L_2^0(H))$ the space of all functions $\sigma : H \to L_2^0(H)$ such that for some $n \in \bbn$ we have $\sigma^j(H) \subset E_n$ for all $j \in \bbn$.
\end{definition}

\begin{definition}
We denote by $\G(H,L_2^0(H))$ the space of all functions $\sigma : H \to L_2^0(H)$ such that for some index $N \in \bbn$ we have $\sigma^j = 0$ for all $j \in \bbn$ with $j > N$.
\end{definition}

\begin{remark}
In view of the following Lemma \ref{lemma-rho} and later results such as Lemma \ref{lemma-norm-Tn}, we emphasize that for a bounded linear operator $T$ we denote by $\| T \|$ the usual operator norm. 
As an exception, we agree that in the particular situation $\sigma \in L_2^0(H)$ we denote by $\| \sigma \|$ the Hilbert-Schmidt norm defined in (\ref{norm-HS}), unless stated otherwise.
\end{remark}

\begin{lemma}\label{lemma-rho}
Let $\sigma \in C_b^2(H,L_2^0(H)) \cap \G(H,L_2^0(H))$ be arbitrary. Then the following statements are true:
\begin{enumerate}
\item For each $h \in H$ we have $\sum_{j \in \bbn} \| D \sigma^j(h) \sigma^j(h) \| < \infty$.

\item The function $\rho : H \to H$ defined as
\begin{align}\label{def-rho}
\rho(h) := \frac{1}{2} \sum_{j \in \bbn} D \sigma^j(h) \sigma^j(h), \quad h \in H,
\end{align}
belongs to ${\rm Lip}(H) \cap \B(H)$.
\end{enumerate}
\end{lemma}

\begin{proof}
By assumption, there exists a constant $C \in \bbr_+$ such that
\begin{align*}
\max \{ \| \sigma(h) \|, \| D \sigma(h) \|, \| D^2 \sigma(h) \| \} \leq C \quad \text{for all $h \in H$.}
\end{align*}
Furthermore, there exists an index $N \in \bbn$ such that $\sigma^j(h) = 0$ for all $h \in H$ and all $j \in \bbn$ with $j > N$. Noting that for each $j \in \bbn$ the norm of the linear operator $L_2^0(H) \to H$, $\sigma \mapsto \sigma^j$ is bounded by $1$, by the chain rule and the Cauchy Schwarz inequality, for each $h \in H$ we obtain
\begin{align*}
&\sum_{j \in \bbn} \| D \sigma^j(h) \sigma^j(h) \| \leq \sum_{j=1}^N \| D \sigma^j(h) \| \, \| \sigma^j(h) \|
\\ &\leq \bigg( \sum_{j=1}^N \| D \sigma(h) \|^2 \bigg)^{1/2} \bigg( \sum_{j=1}^N \| \sigma^j(h) \|^2 \bigg)^{1/2} \leq \sqrt{N} \| D \sigma(h) \| \, \| \sigma(h) \| \leq \sqrt{N} C^2.
\end{align*}
proving the first statement and $\rho \in \B(H)$. For the proof of the second statement, let $h_1,h_2 \in H$ be arbitrary. By the chain rule and Cauchy Schwarz inequality we obtain
\begin{align*}
&\| \rho(h_1) - \rho(h_2) \| \leq \frac{1}{2} \sum_{j=1}^N \| D \sigma^j(h_1) \sigma^j(h_1) - D \sigma^j(h_2) \sigma^j(h_2) \|
\\ &\leq \frac{1}{2} \sum_{j=1}^N \| D \sigma^j(h_1) \| \, \| \sigma^j(h_1) - \sigma^j(h_2) \| + \frac{1}{2} \sum_{j=1}^N \| \sigma^j(h_2) \| \, \| D \sigma^j(h_1) - D \sigma^j(h_2) \|
\\ &\leq \frac{1}{2} \bigg( \sum_{j=1}^N \| D \sigma(h_1) \|^2 \bigg)^{1/2} \bigg( \sum_{j=1}^N \| \sigma^j(h_1) - \sigma^j(h_2) \|^2 \bigg)^{1/2}
\\ &\quad + \frac{1}{2} \bigg( \sum_{j=1}^N \| \sigma^j(h_2) \|^2 \bigg)^{1/2} \bigg( \sum_{j=1}^N \| D \sigma(h_1) - D \sigma(h_2) \|^2 \bigg)^{1/2},
\end{align*}
and hence
\begin{align*}
\| \rho(h_1) - \rho(h_2) \| &\leq \frac{\sqrt{N}}{2} \| D \sigma(h_1) \| \, \| \sigma(h_1) - \sigma(h_2) \|
\\ &\quad + \frac{\sqrt{N}}{2} \| \sigma(h_2) \| \, \| D \sigma(h_1) - D \sigma(h_2) \|
\\ &\leq \sqrt{N} C^2 \| h_1 - h_2 \|, 
\end{align*}
showing that $\rho \in {\rm Lip}(H)$.
\end{proof}

\begin{lemma}\label{lemma-rho-parallel}
Let $\sigma \in C_b^2(H,L_2^0(H)) \cap \G(H,L_2^0(H))$ be such that for each $j \in \bbn$ the function $\sigma^j : H \to H$ is weakly locally parallel. Then the function $\rho : H \to H$ defined in (\ref{def-rho}) is parallel.
\end{lemma}

\begin{proof}
Let $(h^*,h) \in D$ be arbitrary. Furthermore, let $j \in \bbn$ be arbitrary. Since $\sigma^j$ is locally parallel, there exists $\epsilon > 0$ such that
\begin{align*}
\langle h^*, \sigma^j(h-g) \rangle = 0 \quad \text{for all $g \in H$ with $\| g \| \leq \epsilon$.}
\end{align*}
We define $\delta > 0$ as
\begin{align*}
\delta :=
\begin{cases}
\epsilon / \| \sigma^j(h) \|, & \text{if $\sigma^j(h) \neq 0$,}
\\ 1, & \text{if $\sigma^j(h) = 0$.}
\end{cases}
\end{align*}
Then we have
\begin{align*}
\langle h^*, \sigma^j(h + t \sigma^j(h)) \rangle = 0 \quad \text{for all $t \in [-\delta,\delta]$.}
\end{align*}
Therefore, we obtain
\begin{align*}
\langle h^*, D \sigma^j(h) \sigma^j(h) \rangle &= \Big\langle h^*, \lim_{t \to 0} \frac{\sigma^j(h + t \sigma^j(h)) - \sigma^j(h)}{t} \Big\rangle
\\ &= \lim_{t \to 0} \frac{\langle h^*,\sigma^j(h + t \sigma^j(h)) \rangle - \langle h^*, \sigma^j(h) \rangle}{t} = 0.
\end{align*}
This implies
\begin{align*}
\langle h^*,\rho(h) \rangle = \bigg\langle h^*, \frac{1}{2} \sum_{j \in \bbn} D \sigma^j(h) \sigma^j(h) \bigg\rangle = \frac{1}{2} \sum_{j \in \bbn} \langle h^*,D \sigma^j(h) \sigma^j(h) \rangle = 0,
\end{align*}
showing that $\rho$ is parallel.
\end{proof}

\begin{lemma}\label{lemma-sigma-u}
Let $\sigma : H \to L_2^0(H)$ be such that for each $j \in \bbn$ the function $\sigma^j : H \to H$ is parallel. Then, for each $u \in U_0$ the function $\sigma(\cdot)u : H \to H$ is parallel.
\end{lemma}

\begin{proof}
Recall that we have fixed an orthonormal basis $\{ g_j \}_{j \in \bbn}$ of $U_0$. Let $u \in U_0$ be arbitrary, and let $(h^*,h) \in D$ be arbitrary. Since for each $j \in \bbn$ the function $\sigma^j : H \to H$ is parallel, we obtain
\begin{align*}
\langle h^*, \sigma(h)u \rangle = \Big\langle h^*, \sigma(h) \sum_{j \in \bbn} \langle u,g_j \rangle_{U_0} g_j \Big\rangle = \sum_{j \in \bbn} \langle u,g_j \rangle_{U_0} \langle h^*, \sigma^j(h) \rangle = 0,
\end{align*}
showing that $\sigma(\cdot)u$ is parallel.
\end{proof}

For each $n \in \bbn$ let $G_n \subset U_0$ be the finite dimensional subspace $G_n := \langle g_1,\ldots,g_n \rangle$, denote by $\pi_n : U_0 \to G_n$ the corresponding projection
\begin{align*}
\pi_n u = \sum_{j=1}^n \langle u,g_j \rangle_{U_0} g_j, \quad u \in U_0,
\end{align*}
and let $T_n : L_2^0(H) \to L_2^0(H)$ be the linear operator given by $T_n \sigma := \sigma \circ \pi_n$ for each $\sigma \in L_2^0(H)$. Note that for each $n \in \bbn$ and each $\sigma \in L_2^0(H)$ we have
\begin{align}\label{finite}
(T_n \sigma)^j = \sigma(\pi_n(g_j)) = \sigma^j \bbI_{\{ j \leq n \}}, \quad j \in \bbn.
\end{align}

\begin{lemma}\label{lemma-norm-Tn}
The following statements are true:
\begin{enumerate}
\item For each $n \in \bbn$ we have $\| T_n \| \leq 1$.

\item For each $\sigma \in L_2^0(H)$ we have $T_n \sigma \to \sigma$ as $n \to \infty$.
\end{enumerate}
\end{lemma}

\begin{proof}
Let $\sigma \in L_2^0(H)$ be arbitrary. Noting (\ref{norm-HS}) and (\ref{finite}), for each $n \in \bbn$ we have
\begin{align*}
\| T_n \sigma \| = \sqrt{\sum_{j=1}^n \| \sigma^j \|^2} \leq \sqrt{\sum_{j \in \bbn} \| \sigma^j \|^2} = \| \sigma \|,
\end{align*}
showing that $\| T_n \| \leq 1$. Furthermore, by (\ref{norm-HS}) and (\ref{finite}) we obtain
\begin{align*}
\| T_n \sigma - \sigma \| = \sqrt{\sum_{j > n} \| \sigma^j \|^2} \to 0 \quad \text{as $n \to \infty$,}
\end{align*}
showing that $T_n \sigma \to \sigma$.
\end{proof}

\begin{proposition}\label{prop-sigma-finite}
Let $\sigma \in {\rm Lip}(H,L_2^0(H))$ be such that for each $j \in \bbn$ the function $\sigma^j : H \to H$ is parallel. Then there are a constant $L \in \bbr_+$ and a sequence
\begin{align}\label{sigma-finite}
(\sigma_n)_{n \in \bbn} \subset {\rm Lip}_L(H,L_2^0(H)) \cap \G(H,L_2^0(H))
\end{align}
such that for all $n,j \in \bbn$ the function $\sigma_n^j : H \to H$ is parallel, and we have $\sigma_n \to \sigma$.
\end{proposition}

\begin{proof}
We set $\sigma_n := T_n \circ \sigma$ for each $n \in \bbn$. By noting (\ref{finite}), we have $(\sigma_n)_{n \in \bbn} \subset \G(H,L_2^0(H))$, and for all $n,j \in \bbn$ the function $\sigma_n^j : H \to H$ is parallel. By hypothesis, there is a constant $L \in \bbr_+$ such that $\sigma \in {\rm Lip}_L(H,L_2^0(H))$, and by Lemma \ref{lemma-norm-Tn}, it follows that $\sigma_n \in {\rm Lip}_L(H,L_2^0(H))$ for each $n \in \bbn$, showing (\ref{sigma-finite}), and that $\sigma_n \to \sigma$.
\end{proof}

\begin{lemma}\label{lemma-Pi-sigma}
Let $\sigma : H \to H$ be a parallel function. Then, for each $n \in \bbn$ the function $\Pi_n \circ \sigma$ is parallel, too.
\end{lemma}

\begin{proof}
Let $(h^*,h) \in D$ be arbitrary. By Assumption \ref{ass-Schauder-basis} we have $h^* \in \langle e_k^* \rangle$ for some $k \in \bbn$. Therefore, by Lemma \ref{lemma-proj-par} we obtain
\begin{align*}
\langle h^*,\Pi_n(\sigma(h)) \rangle = \langle h^*,\sigma(h) \rangle \mathbbm{1}_{\{ k \leq n \}} = 0,
\end{align*}
finishing the proof.
\end{proof}

In view of the following results, recall the Definition \ref{def-F-spaces} of $\F(H)$.

\begin{proposition}\label{prop-sigma-FDR}
Let $\sigma \in {\rm Lip}(H)$ be a parallel function. Then there are a constant $L \in \bbr_+$ and a sequence 
\begin{align}\label{sigma-FDR}
(\sigma_n)_{n \in \bbn} \subset \Lip_L(H) \cap \F(H)
\end{align}
such that $\sigma_n$ is parallel for each $n \in \bbn$, and we have $\sigma_n \to \sigma$.
\end{proposition}

\begin{proof}
We set $\sigma_n := \Pi_n \circ \sigma$ for each $n \in \bbn$. Then, by construction for each $n \in \bbn$ we have $\sigma_n \in \F(H)$. By hypothesis there exists a constant $M \in \bbr_+$ such that $\sigma \in {\rm Lip}_M(H)$. Setting $L := M \bc$, we have $\sigma_n \in {\rm Lip}_L(H)$ for each $n \in \bbn$, showing (\ref{sigma-FDR}). Furthermore, by Lemma \ref{lemma-Pi-sigma}, for each $n \in \bbn$ the function $\sigma_n$ is parallel, and by Lemma \ref{lemma-proj-par} we have $\sigma_n \to \sigma$.
\end{proof}

\begin{lemma}\label{lemma-inward-comp-2}
Let $\sigma,\tau : H \to H$ be two functions such that the following conditions are fulfilled:
\begin{enumerate}
\item $\sigma$ is parallel.

\item $D$ is $({\rm Id}_H,\tau)$-invariant.
\end{enumerate}
Then $\sigma \circ \tau$ is parallel.
\end{lemma}

\begin{proof}
Let $(h^*,h) \in D$ be arbitrary. Then we have $(h^*,\tau(h)) \in D$, because $D$ is $({\rm Id}_H,\tau)$-invariant. Therefore, and since $\sigma$ is parallel, we obtain
\begin{align*}
\langle h^*,\sigma(\tau(h)) \rangle = 0,
\end{align*}
finishing the proof.
\end{proof}

\begin{proposition}\label{prop-sigma-bounded}
Let $\sigma \in {\rm Lip}(H) \cap \F(H)$ be a parallel function. Then there are a constant $L \in \bbr_+$ and a sequence
\begin{align}\label{sigma-bounded}
(\sigma_n)_{n \in \bbn} \subset \Lip_L(H) \cap \F(H) \cap \B(H)
\end{align}
such that $\sigma_n$ is parallel for each $n \in \bbn$, and we have $\sigma_n \to \sigma$.
\end{proposition}

\begin{proof}
We set $\sigma_n := \sigma \circ R_n$ for each $n \in \bbn$. Let $n \in \bbn$ be arbitrary. Then we have $\sigma_n \in \F(H)$, because $\sigma \in \F(H)$. By hypothesis there exists a constant $L \in \bbr_+$ such that $\sigma \in {\rm Lip}_L(H)$, and by Lemma \ref{lemma-retract-X} and the inclusion $\Lip_L(H) \subset \B^{\loc}(H)$ it follows that $\sigma_n \in {\rm Lip}_L(H) \cap \B(H)$, showing (\ref{sigma-bounded}). Combining Lemmas \ref{lemma-inward-comp-2} and \ref{lemma-retract-par}, we obtain that $\sigma_n$ is parallel. Furthermore, by Lemma \ref{lemma-retract-X} we have $\sigma_n \to \sigma$.
\end{proof}

\begin{lemma}\label{lemma-tau-sigma-local}
Let $\sigma,\tau : H \to H$ be two functions such that the following conditions are fulfilled:
\begin{enumerate}
\item $\sigma$ is parallel.

\item $D$ is locally $({\rm Id}_H,\tau)$-invariant.
\end{enumerate}
Then $\sigma \circ \tau$ is locally parallel.
\end{lemma}

\begin{proof}
By assumption, there exists $\epsilon > 0$ such that for all $(h^*,h) \in D$ we have
\begin{align*}
(h^*,\tau(h-g)) \in D \quad \text{for all $g \in H$ with $\| g \| \leq \epsilon$.}
\end{align*}
Let $(h^*,h) \in D$ be arbitrary. Since $\sigma$ is parallel, we obtain
\begin{align*}
\langle h^*, \sigma(\tau(h-g)) \rangle = 0 \quad \text{for all $g \in H$ with $\| g \| \leq \epsilon$,}
\end{align*}
completing the proof.
\end{proof}

For $\epsilon > 0$ let $\phi_{\epsilon} : \bbr \to \bbr$ be the function given by (\ref{def-psi-intro});
see Figure \ref{fig-approx}. Then we have $\phi_{\epsilon} \in {\rm Lip}_1(\bbr)$ and
\begin{align}
\phi_{\epsilon}(x) &= 0 \quad \text{for all $x \in [-\epsilon,\epsilon]$,}
\\ \phi_{\epsilon}(x) &\geq 0 \quad \text{for all $x \in [-\epsilon,\infty)$,}
\\ \label{psi-2} | \phi_{\epsilon}(x) - x | &\leq \epsilon \quad \text{for all $x \in \bbr$,}
\\ \label{psi-3} \bigg| \frac{\phi_{\epsilon}(x) - \phi_{\epsilon}(y)}{x-y} \bigg| &\leq 1 \quad \text{for all $x,y \in \bbr$ with $x \neq y$.}
\end{align}
Furthermore, for each $\theta \in \{ -1,1 \}$ we have
\begin{align}\label{psi-theta-1}
\theta \phi_{\epsilon}(\theta y) &\geq 0 \quad \text{for all $y \in [-\epsilon, \infty)$,}
\\ \label{psi-theta-2} x - \theta \phi_{\epsilon}(\theta y) &\geq 0 \quad \text{for all $x \in \bbr_+$ and $y \in \bbr$ with $|x-y| \leq \epsilon$.}
\end{align}

\begin{lemma}\label{lemma-Psi-n}
There exist a constant $L \in \bbr_+$ and a sequence $(\Phi_n)_{n \in \bbn} \subset {\rm Lip}_L(H)$ such that for each $n \in \bbn$ the set $D$ is locally $({\rm Id}_H,\Phi_n)$-invariant, and we have $\Phi_n \to {\rm Id}_H$.
\end{lemma}

\begin{proof}
We set $L := 2 \ubc$. Let $n \in \bbn$ be arbitrary. We define the function
\begin{align}\label{def-Psi-n}
\Phi_n : H \to H, \quad \Phi_n(h) := \sum_{k=1}^n \phi_{2^{-n}}(h_k) e_k,
\end{align}
where we refer to the series representation (\ref{series-h}) of $h$. Let $h,g \in H$ be arbitrary. We define the sequence $(\lambda_k)_{k \in \bbn} \subset \bbr$ as
\begin{align*}
\lambda_k := \frac{\phi_{2^{-n}}(h_k)-\phi_{2^{-n}}(g_k)}{h_k - g_k} \mathbbm{1}_{\{ h_k \neq g_k \}} \bbI_{\{ k \leq n \}}, \quad k \in \bbn.
\end{align*}
By (\ref{psi-3}) we have $|\lambda_k| \leq 1$ for all $k \in \bbn$, and  by Lemma \ref{lemma-norm-g-star} we obtain
\begin{align*}
\| \Phi_n(h) - \Phi_n(g) \| &= \bigg\| \sum_{k=1}^n (\phi_{2^{-n}}(h_k) - \phi_{2^{-n}}(g_k)) e_k \bigg\| 
\\ &= \bigg\| \sum_{k \in \bbn} \lambda_k (h_k - g_k) e_k \bigg\| \leq L \bigg\| \sum_{k \in \bbn} (h_k - g_k) e_k \bigg\| = L \| h-g \|,
\end{align*}
showing that $\Phi_n \in {\rm Lip}_L(H)$. Let $h \in H$ be arbitrary. Then, by (\ref{psi-2}) we obtain
\begin{align*}
&\| \Phi_n(h) - h \| = \bigg\| \sum_{k = 1}^n \phi_{2^{-n}}(h_k) e_k - \sum_{k \in \bbn} h_k e_k \bigg\|
\\ &= \bigg\| \sum_{k = 1}^n (\phi_{2^{-n}}(h_k) - h_k) e_k - \sum_{k = n+1}^{\infty} h_k e_k \bigg\|
\leq \sum_{k=1}^n | \phi_{2^{-n}}(h_k) - h_k | + \bigg\| \sum_{k = n+1}^{\infty} h_k e_k \bigg\|
\\ &\leq n \cdot 2^{-n} + \bigg\| \sum_{k = n+1}^{\infty} h_k e_k \bigg\| \to 0 \quad \text{a $n \to \infty$,}
\end{align*}
showing that $\Phi_n \to {\rm Id}_H$. Let $n \in \bbn$ be arbitrary. In order to show that $D$ is locally $({\rm Id}_H,\Phi_n)$-invariant, we set $\epsilon := 2^{-n} / L$.
Let $(h^*,h) \in D$ be arbitrary, and let $g \in H$ with $\| g \| \leq \epsilon$ be arbitrary. We will show that $(h^*,\Phi_n(h-g)) \in D$. For this purpose, let $g^* \in G^*$ be arbitrary. Since $\| g \| \leq \epsilon$, by Lemma \ref{lemma-norm-g-star} we have
\begin{align}\label{h-g-rel}
|\langle g^*,g \rangle| \leq L \| g \| \leq L \epsilon = 2^{-n}.
\end{align}
Since $h \in K$, we have $\langle g^*,h \rangle \geq 0$, and hence, we obtain
\begin{align}\label{g-g-rel}
\langle g^*,h-g \rangle = \langle g^*,h \rangle - \langle g^*,g \rangle \geq -L\epsilon = -2^{-n}.
\end{align}
By Assumption \ref{ass-Schauder-basis} we have $g^* = \theta e_k^*$ for some $\theta \in \{ -1,1 \}$ and some $k \in \bbn$. Thus, by the definition (\ref{def-Psi-n}) of $\Phi_n$ and relations (\ref{g-g-rel}) and (\ref{psi-theta-1}) we deduce
\begin{align*}
\langle g^*,\Phi_n(h-g) \rangle = \theta \phi_{2^{-n}}(\theta \langle g^*,h-g \rangle) \bbI_{\{ k \leq n \}} \geq 0,
\end{align*}
showing that $\Phi_n(h-g) \in K$. Furthermore, noting that $h \in K$, by the definition (\ref{def-Psi-n}) of $\Phi_n$ and relations (\ref{h-g-rel}) and (\ref{psi-theta-2}) we obtain
\begin{align*}
\langle g^*,h-\Phi_n(h-g) \rangle =  \langle g^*,h \rangle - \theta \phi_{2^{-n}}(\theta \langle g^*,h-g \rangle) \bbI_{\{ k \leq n \}} \geq 0,
\end{align*}
showing that $h-\Phi_n(h-g) \in K$, and hence $\Phi_n(h-g) \leq_K h$. By Assumption \ref{ass-D} we deduce that $(h^*,\Phi_n(h-g)) \in D$, showing that $D$ is locally $({\rm Id}_H,\Phi_n)$-invariant.
\end{proof}

\begin{proposition}\label{prop-locally-par}
Let $\sigma \in {\rm Lip}(H) \cap \F(H) \cap \B(H)$ be a parallel function. Then there are a constant $L \in \bbr$ and a sequence
\begin{align}\label{locally-par}
(\sigma_n)_{n \in \bbn} \subset {\rm Lip}_L(H) \cap \F(H) \cap \B(H)
\end{align}
such that $\sigma_n$ is locally parallel for each $n \in \bbn$, and we have $\sigma_n \to \sigma$.
\end{proposition}

\begin{proof}
According to Lemma \ref{lemma-Psi-n}, there exist a constant $M \in \bbr$ and a sequence $(\Phi_n)_{n \in \bbn} \subset {\rm Lip}_M(H)$ such that for each $n \in \bbn$ the set $D$ is locally $({\rm Id}_H,\Phi_n)$-invariant, and we have $\Phi_n \to {\rm Id}_H$. Therefore, setting $\sigma_n := \sigma \circ \Phi_n$ for each $n \in \bbn$, we have (\ref{locally-par}) for some $L \in \bbr$, and applying Lemma \ref{lemma-tau-sigma-local} shows that $\sigma_n$ is locally parallel for each $n \in \bbn$.
\end{proof}

For our next step, we apply the sup-inf convolution technique from \cite{Lasry-Lions}.

\begin{definition}
Let $\sigma : H \to \bbr$ be arbitrary.
\begin{enumerate}
\item For each $\lambda > 0$ we define
\begin{align*}
\sigma_{\lambda} : H \to \bbr, \quad \sigma_{\lambda}(h) := \inf_{g \in H} \bigg( \sigma(g) + \frac{1}{2 \lambda} \| h-g \|^2 \bigg).
\end{align*}
\item For each $\mu > 0$ we define
\begin{align*}
\sigma^{\mu} : H \to \bbr, \quad \sigma^{\mu}(h) := \sup_{g \in H} \bigg( \sigma(g) - \frac{1}{2 \mu} \| h-g \|^2 \bigg).
\end{align*}
\end{enumerate}
\end{definition}

\begin{remark}
Let $\sigma : H \to \bbr$ and $\lambda,\mu > 0$ be arbitrary. A straightforward calculation shows that
\begin{align*}
(\sigma_{\lambda})^{\mu}(h) = \sup_{f \in H} \inf_{g \in H} \bigg( \sigma(g) + \frac{1}{2 \lambda} \| f-g \|^2 - \frac{1}{2 \mu} \| f-h \|^2 \bigg) \quad \text{for all $h \in H$.}
\end{align*}
Therefore, the function $(\sigma_{\lambda})^{\mu}$ is also called \emph{sup-inf convolution}.
\end{remark}

\begin{definition}
Let $\sigma \in \F(H)$ be arbitrary.
\begin{enumerate}
\item For each $\lambda > 0$ we define $\sigma_{\lambda} : H \to H$ as
\begin{align*}
\sigma_{\lambda} := \sum_{k \in \bbn} (\sigma_k)_{\lambda} e_k.
\end{align*}
\item For each $\mu > 0$ we define $\sigma^{\mu} : H \to H$ as
\begin{align*}
\sigma^{\mu} := \sum_{k \in \bbn} (\sigma_k)^{\mu} e_k.
\end{align*}
\item For all $\lambda,\mu > 0$ we define $(\sigma_{\lambda})^{\mu} : H \to H$ as
\begin{align*}
(\sigma_{\lambda})^{\mu} := \sum_{k \in \bbn} ((\sigma_k)_{\lambda})^{\mu} e_k.
\end{align*}
\end{enumerate}
\end{definition}

\begin{lemma}\label{lemma-sup-inf-2}
Let $\sigma \in {\rm Lip}_L(H) \cap \F(H) \cap \B(H)$ be arbitrary. Then, for each $\epsilon > 0$ there are $\lambda_0,\mu_0 > 0$ such that for all $\lambda \in (0,\lambda_0]$ and $\mu \in (0,\mu_0]$ with $\mu < \lambda$ we have
\begin{align*}
\sup_{h \in H} \| (\sigma_{\lambda})^{\mu}(h) - \sigma(h) \| \leq \epsilon.
\end{align*}
\end{lemma}

\begin{proof}
This follows from the theorem on pages 260, 261 in \cite{Lasry-Lions}; in particular relation (12) therein.
\end{proof}

\begin{lemma}\label{lemma-Lip-components-pre}
There is a constant $C \in \bbr_+$ such that for all $L \in \bbr_+$ and all $\sigma \in \Lip_L(H)$ we have $\sigma_k \in \Lip_{CL}(H,\bbr)$ for each $k \in \bbn$.
\end{lemma}

\begin{proof}
Setting $C := 2 \bc$, this is an immediate consequence of Lemma \ref{lemma-norm-g-star}.
\end{proof}

\begin{lemma}\label{lemma-Lip-components}
Let $L \in \bbr_+$ and $\sigma \in \F(H)$ be such that $\sigma_k \in \Lip_L(H,\bbr_+)$ for all $k = 1,\ldots,N$, where $N := \dim \langle \sigma(H) \rangle$. Then we have $\sigma \in \Lip_{NL}(H)$.
\end{lemma}

\begin{proof}
For all $h,g \in H$ we have
\begin{align*}
\| \sigma(h) - \sigma(g) \| = \bigg\| \sum_{k=1}^N (\sigma_k(h) - \sigma_k(g)) e_k \bigg\| \leq \sum_{k=1}^N |\sigma_k(h) - \sigma_k(g)| \leq NL \| h-g \|,
\end{align*}
completing the proof.
\end{proof}

\begin{lemma}\label{lemma-sup-inf-1}
There exists a constant $C \in \bbr_+$ such that for all $L \in \bbr_+$, all $\sigma \in {\rm Lip}_L(H) \cap \F(H) \cap \B(H)$ and all $\lambda,\mu > 0$ with $\mu < \lambda$ we have
\begin{align*}
(\sigma_{\lambda})^{\mu} \in {\rm Lip}_{CNL}(H) \cap \F(H) \cap C_b^{1,1}(H),
\end{align*}
where $N := \dim \langle \sigma(H) \rangle$.
\end{lemma}

\begin{proof}
Let $\lambda,\mu > 0$ with $\mu < \lambda$ be arbitrary. For all $k \in \bbn$ with $\sigma_k = 0$ we have $((\sigma_k)_{\lambda})^{\mu} = 0$, showing that $(\sigma_{\lambda})^{\mu} \in \F(H)$. The remaining assertions follow from Lemmas \ref{lemma-Lip-components-pre}, \ref{lemma-Lip-components} and the theorem on pages 260, 261 in \cite{Lasry-Lions}; in particular relations (11), (13) and (15) therein.
\end{proof}

\begin{lemma}\label{lemma-sup-inf-parallel}
Let $\sigma \in {\rm Lip}(H) \cap \F(H) \cap \B(H)$ be a locally parallel function. Then the following statements are true:
\begin{enumerate}
\item There exists $\lambda_0 > 0$ such that $\sigma_{\lambda}$ is locally parallel for each $\lambda \in (0,\lambda_0]$.

\item There exists $\mu_0 > 0$ such that $\sigma^{\mu}$ is locally parallel for each $\mu \in (0,\mu_0]$.

\item There exist $\lambda_0,\mu_0 > 0$ such that $(\sigma_{\lambda})^{\mu}$ is locally parallel for all $\lambda \in (0,\lambda_0]$ and $\mu \in (0,\mu_0]$ with $\mu < \lambda$.
\end{enumerate}
\end{lemma}

\begin{proof}
Since $\sigma$ is locally parallel, there exists $\epsilon > 0$ such that for all $(h^*,h) \in D$ we have (\ref{loc-par-prop}).
Furthermore, since $\sigma \in \B(H)$, there exists a finite constant $C > 0$ such that                                                               \begin{align}\label{sup-inf-sigma-bounded}
\| \sigma(h) \| \leq C \quad \text{for all $h \in H$.}
\end{align}
We define the constants $M,\lambda_0 > 0$ as
\begin{align*}
M := 2 C \bc \quad \text{and} \quad \lambda_0 := \frac{\epsilon^2}{8M}.
\end{align*}
Let $\lambda \in (0,\lambda_0]$ be arbitrary. We will show that $\sigma_{\lambda}$ is locally parallel. For this purpose, let $(h^*,h) \in D$ be arbitrary. By Assumption \ref{ass-Schauder-basis} there exist $\theta \in \{ -1,1 \}$ and $k \in \bbn$ such that $h^* = \theta e_k^*$. Let $g \in H$ with $\| g \| \leq \epsilon / 2$ be arbitrary. We define the function
\begin{align*}
\Sigma : H \to \bbr, \quad \Sigma(f) := \sigma_k(f) + \frac{1}{2 \lambda} \| (h-g)-f \|^2.
\end{align*}
Then we have
\begin{align}\label{sup-inf-proof-3}
\Sigma \geq 0 \quad \text{and} \quad \Sigma(h-g) = 0.
\end{align}
Indeed, by (\ref{loc-par-prop}) we have $\sigma_k(h-g) = 0$, and hence $\Sigma(h-g) = 0$. In order to show that $\Sigma \geq 0$, let $f \in H$ be arbitrary. We distinguish two cases:
\begin{itemize}
\item Suppose that $\| h - f \| \leq \epsilon$. Since $f = h - (h-f)$, by (\ref{loc-par-prop}) we have $\sigma_k(f) = 0$, showing $\Sigma(f) \geq 0$.

\item Suppose that $\| h - f \| > \epsilon$. Since $\| g \| \leq \epsilon / 2$, by the inverse triangle inequality we obtain
\begin{align*}
\| (h-g) - f \| = \| (h-f) - g \| \geq | \, \| h-f \| - \| g \| \, | \geq \epsilon / 2.
\end{align*}
Furthermore, by (\ref{sup-inf-sigma-bounded}) and Lemma \ref{lemma-norm-g-star} we have
\begin{align*}
| \sigma_k | = |\langle e_k^*,\sigma \rangle| \leq 2 \bc \| \sigma \| \leq M,
\end{align*}
and hence
\begin{align*}
\Sigma(f) &= \sigma_k(f) + \frac{1}{2 \lambda} \| (h-g)-f \|^2 \geq -M + \frac{1}{2 \lambda_0} \frac{\epsilon^2}{4} = 0.
\end{align*}
\end{itemize}
Consequently, we have (\ref{sup-inf-proof-3}), and thus, we obtain
\begin{align*}
\langle h^*, \sigma_{\lambda}(h-g) \rangle = \theta \inf_{f \in H} \Sigma(f) = 0,
\end{align*}
showing that $\sigma_{\lambda}$ is locally parallel.
This provides the proof of the first statement. The proof of the second statement is analogous, and the third statement follows from the first and the second statement.
\end{proof}

\begin{proposition}\label{prop-C-b-1-1}
Let $\sigma \in {\rm Lip}(H) \cap \F(H) \cap \B(H)$ be a locally parallel function. Then there are a constant $L \in \bbr_+$ and a sequence
\begin{align*}
(\sigma_n)_{n \in \bbn} \subset \Lip_L(H) \cap \F(H) \cap C_b^{1,1}(H)
\end{align*}
such that $\sigma_n$ is locally parallel for each $n \in \bbn$, and we have $\sigma_n \to \sigma$.
\end{proposition}

\begin{proof}
This is an immediate consequence of Lemmas \ref{lemma-sup-inf-2}, \ref{lemma-sup-inf-1} and \ref{lemma-sup-inf-parallel}.
\end{proof}

For our last step, we use Moulis' method, as presented in \cite{Fry}. For this purpose, we introduce some notation. Let $\varphi \in C^{\infty}(\bbr,[0,1])$ be a smooth function such that the following conditions are fulfilled:
\begin{itemize}
\item We have $\varphi(t) = 1$ for all $t \in (-\frac{1}{2},\frac{1}{2})$.

\item We have $\varphi(t) = 0$ for all $t \in \bbr$ with $|t| \geq 1$.

\item We have $\varphi'(t) \in [-3,0]$ for all $t \in \bbr_+$.

\item We have $\varphi(-t) = \varphi(t)$ for all $t \in \bbr_+$.
\end{itemize}
Let $\sigma \in \F(H) \cap C_b^{1,1}(H)$ be arbitrary. We fix a sequence $a = (a_n)_{n \in \bbn} \subset (0,\infty)$ and a constant $r > 0$. We define the sequence $(\Sigma_n)_{n \in \bbn}$ of functions $\Sigma_n : H \to H$ as
\begin{align}\label{def-cap-sigma}
\Sigma_n(h) := \frac{(a_n)^n}{c_n} \int_{E_n} \sigma(h-g) \varphi(a_n \| g \|) dg, \quad h \in H,
\end{align}
where the sequence $(c_n)_{n \in \bbn} \subset (0,\infty)$ is given by
\begin{align}\label{def-c-n-const}
c_n := \int_{E_n} \varphi(\| g \|) dg.
\end{align}

\begin{lemma}\label{lemma-cap-sigma}
The following statements are true:
\begin{enumerate}
\item We have $\Sigma_n \in C^{\infty}(H)$ for each $n \in \bbn$.

\item There is a constant $C \in \bbr_+$ such that
\begin{align}\label{cap-sigma-n-est}
\max \{ \| \Sigma_n(h) \|, \| D \Sigma_n(h) \|, \| D^2 \Sigma_n(h) \| \} \leq C
\end{align}
for all $n \in \bbn$ and all $h \in H$.
\end{enumerate}
\end{lemma}

\begin{proof}
The first statement follows from the definition (\ref{def-cap-sigma}). Since $\sigma \in C_b^{1,1}(H)$, there is a constant $C \in \bbr_+$ such that
\begin{align*}
\max \{ \| \sigma(h) \| + \| D \sigma(h) \| \} &\leq C \quad \text{for all $h \in H$,}
\\ \| D \sigma(h) - D \sigma(g) \| &\leq C \| h-g \| \quad \text{for all $h,g \in H$.}
\end{align*}
Thus, arguing as in \cite[page 602]{Fry}, we see that (\ref{cap-sigma-n-est}) is fulfilled.
\end{proof}

Now, we define the sequence $(\hat{\sigma}_n)_{n \in \bbn}$ of functions $\hat{\sigma}_n : H \to H$ as
\begin{align}\label{def-hat-sigma}
\hat{\sigma}_n(h) := \frac{(b_n)^n}{c_n} \int_{E_n} \Sigma_n(h-g) \varphi(b_n \| g \|) dg, \quad h \in H,
\end{align}
where the sequence $b = (b_n)_{n \in \bbn} \subset (0,\infty)$ is chosen large enough such that
\begin{align}\label{diff-hat-cap}
\max \{ \| \hat{\sigma}_n(h) - \Sigma_n(h) \|, \| D \hat{\sigma}_n(h) - D \Sigma_n(h) \|, \| D^2 \hat{\sigma}_n(h) - D^2 \Sigma_n(h) \| \} \leq 2^{-n}
\end{align}
for all $n \in \bbn$ and all $h \in H$. Inductively, we define the sequence $(\bar{\sigma}_n)_{n \in \bbn_0}$ of functions $\bar{\sigma}_n : H \to H$ by
\begin{align}\label{def-sigma-0-bar}
\bar{\sigma}_0 &:= \sigma(0) \quad \text{and}
\\ \label{def-sigma-n-bar} \bar{\sigma}_n &:= \hat{\sigma}_n + \bar{\sigma}_{n-1} \circ \Pi_{n-1} - \hat{\sigma}_n \circ \Pi_{n-1} \quad \text{for all $n \in \bbn$.}
\end{align}

\begin{lemma}\label{lemma-bar-sigma}
The following statements are true:
\begin{enumerate}
\item We have $\bar{\sigma}_n|_{E_n} = \bar{\sigma}_{n-1}|_{E_n}$ and $\bar{\sigma}_n|_{E_n} \in C^{\infty}(E_n,H)$ for all $n \in \bbn$.

\item There is a constant $C \in \bbr_+$ such that
\begin{align}\label{est-bar-sigma}
\max \{ \| \bar{\sigma}_n(h) \|, \| D \bar{\sigma}_n(h) \|, \| D^2 \bar{\sigma}_n(h) \| \} \leq C
\end{align}
for all $n \in \bbn$ and all $h \in E_n$.
\end{enumerate}
\end{lemma}

\begin{proof}
The first statement follows from \cite[page 602]{Fry}. Using (\ref{diff-hat-cap}), we prove inductively as in \cite{Fry} that 
\begin{align*}
\max \{ \| \bar{\sigma}_n(h) - \Sigma_n(h) \|, \| D \bar{\sigma}_n(h) - D \Sigma_n(h) \|, \| D^2 \bar{\sigma}_n(h) - D^2 \Sigma_n(h) \| \} \leq 2 (1 - 2^{-n})
\end{align*}
for all $n \in \bbn$ and all $h \in H$. Together with Lemma \ref{lemma-cap-sigma}, this proves the second statement.
\end{proof}

Now, we define $\bar{\sigma} : E^{\infty} \to H$ as
\begin{align}\label{def-sigma-bar}
\bar{\sigma} := \lim_{n \to \infty} \bar{\sigma}_n,
\end{align}
where $E^{\infty} := \bigcup_{n \in \bbn} E_n$. In view of Lemma \ref{lemma-bar-sigma}, we have
\begin{align}\label{restr-bar-sigma}
\bar{\sigma}|_{E_n} = \bar{\sigma}_n|_{E_n} \quad \text{for all $n \in \bbn$.}
\end{align}
Now, we define the function
\begin{align}\label{def-Psi-final}
\Psi : H \to E^{\infty}, \quad \Psi(h) := \sum_{k \in \bbn} \chi_k(h) h_k e_k,
\end{align}
where we refer to the series representation (\ref{series-h}) of $h$, and where for each $k \in \bbn$ the function $\chi_k : H \to [0,1]$ is given by
\begin{align}\label{def-chi}
\chi_k(h) := 1 - \varphi(\| T_k h \|),
\end{align}
where $T_k \in L(H)$ denotes the linear operator
\begin{align}\label{def-T-k}
T_k := \frac{\Id_H - \Pi_{k-1}}{r}
\end{align}
with $r > 0$ denoting the constant from above.

\begin{lemma}\label{lemma-Psi-local}
The following statements are true:
\begin{enumerate}
\item We have $\Psi \in \Lip(H,E^{\infty}) \cap C^{\infty}(H,E^{\infty})$.

\item For each $h \in H$ there exist $n \in \bbn$ and $\delta > 0$ such that
\begin{align}\label{Psi-local}
\Psi(h-g) \in E_n \quad \text{for all $g \in H$ with $\| g \| \leq \delta$.}
\end{align}
\end{enumerate}
\end{lemma}

\begin{proof}
This follows from \cite[page 17]{Fine-Approx}.
\end{proof}

Now, we define the function
\begin{align}\label{def-sigma-final}
\sigma^{(a,b,r)} : H \to H, \quad \sigma^{(a,b,r)} := \bar{\sigma} \circ \Psi.
\end{align}
Note that we emphasize the dependence on the sequences $a$ and $b$, and on the constant $r$. For two sequences $a = (a_n)_{n \in \bbn} \subset \bbr$ and $b = (b_n)_{n \in \bbn} \subset \bbr$ we agree to write $a \leq_{\bbn} b$ if $a_n \leq b_n$ for all $n \in \bbn$.

\begin{lemma}\label{lemma-approx-C-b-2-a}
Let $\sigma \in \Lip(H) \cap \F(H) \cap \B(H)$ be arbitrary. Then, for each $\epsilon > 0$ there are sequences $a^0, b^0 \in (0,\infty)^{\bbn}$, where $b^0$ is chosen such that (\ref{diff-hat-cap}) is fulfilled with $b$ replaced by $b^0$, and a constant $r^0 > 0$ such that for all sequences $a,b \in (0,\infty)^{\bbn}$ with $a^0 \leq_{\bbn} a$ and $b^0 \leq_{\bbn} b$ and all $r > 0$ with $r \leq r^0$ we have 
\begin{align*}
\sup_{h \in H} \| \sigma^{(a,b,r)}(h) - \sigma(h) \| \leq \epsilon.
\end{align*}
\end{lemma}

\begin{proof}
This follows from \cite[Thm. 1]{Fry} and its proof.
\end{proof}

\begin{lemma}\label{lemma-C-b-2-pre}
There exists a constant $C \in \bbr_+$ such that for all $L \in \bbr_+$, all $\sigma \in {\rm Lip}_L(H) \cap \F(H) \cap \B(H)$ and all sequences $a,b \in (0,\infty)^{\bbn}$, where $b$ is chosen such that (\ref{diff-hat-cap}) is fulfilled, and every constant $r > 0$ we have
\begin{align*}
\sigma^{(a,b,r)} \in {\rm Lip}_{CNL}(H) \cap \F(H) \cap C^{\infty}(H),
\end{align*}
where $N := \dim \langle \sigma(H) \rangle$.
\end{lemma}

\begin{proof}
Let $a,b$ be arbitrary sequences, where $b$ is chosen such that (\ref{diff-hat-cap}) is fulfilled, and let $r > 0$ be arbitrary. By the construction (\ref{def-cap-sigma})--(\ref{def-sigma-final}), for all $k \in \bbn$ with $\sigma_k = 0$ we have $\sigma_k^{(a,b,r)} = 0$, showing that $\sigma^{(a,b,r)} \in \F(H)$. The remaining assertions follow from Lemmas \ref{lemma-Lip-components-pre}, \ref{lemma-Lip-components} and \cite[Thm. 1]{Fry}.
\end{proof}

Lemma \ref{lemma-C-b-2-pre} does not ensure that $\sigma^{(a,b,r)} \in C_b^2(H)$; that is, it remains to show that the second order derivative is bounded. For this purpose, we prepare some auxiliary results. For the next two results, we fix a constant $r > 0$. Note that the functions $\chi_k$, $k \in \bbn$ defined in (\ref{def-chi}) and $\Psi$ defined in (\ref{def-Psi-final}) depend on the choice of $r$.

\begin{lemma}\label{lemma-chi-est}
The following statements are true:
\begin{enumerate}
\item We have $\chi_k \in C^{\infty}(H,\bbr)$ for each $k \in \bbn$.

\item There is a constant $C \in \bbr_+$ such that
\begin{align}\label{est-chi-k}
\max \{ \| \chi_k(h) \|, r \| D \chi_k(h) \|, r^2 \| D^2 \chi_k(h) \| \} \leq C
\end{align}
for all $k \in \bbn$ and all $h \in H$.
\end{enumerate}
\end{lemma}

\begin{proof}
Let $U \subset H$ be the open set $U := \{ \| \cdot \| > \frac{1}{4} \}$. For the norm function $\eta : U \to \bbr_+$ given by $\eta(h) := \| h \|$ we have $\eta \in C^{\infty}(U,\bbr)$ with derivatives
\begin{align*}
D\eta(h)g &= \frac{\langle h,g \rangle}{\| h \|}, \quad \text{$h \in U$ and $g \in H$,} 
\\ D^2 \eta(h)(g,f) &= \frac{\langle g,f \rangle}{\| h \|} - \frac{\langle h,g \rangle \langle h,f \rangle}{\| h \|^3}, \quad \text{$h \in U$ and $g,f \in H$.} 
\end{align*}
Therefore, for all $h \in U$ we obtain
\begin{align}\label{norm-est}
\| D \eta(h) \| \leq 1 \quad \text{and} \quad
\| D^2 \eta(h) \| &\leq \frac{2}{\| h \|} \leq 8.
\end{align}
We define the constant $L \in \bbr_+$ as
\begin{align*}
L := 1 + \bc.
\end{align*}
Then, by the definition (\ref{def-T-k}) of $T_k$ we have 
\begin{align}\label{T-k-est}
\| T_k \| \leq L/r \quad \text{for all $k \in \bbn$.} 
\end{align}
There is a constant $M \in \bbr_+$ such that
\begin{align*}
\max \{ \varphi(t), \varphi'(t), \varphi''(t) \} \leq M \quad \text{for all $t \in \bbr$.} 
\end{align*}
Now, we define the constant $C \in \bbr_+$ as
\begin{align*}
C := \max \{ 1, ML, ML^2 + 8 M^2 L^2 \}.
\end{align*}
Let $k \in \bbn$ be arbitrary. By the definition (\ref{def-chi}) of $\chi_k$ we have
\begin{align*}
\chi_k = 1 - \varphi \circ \eta \circ T_k,
\end{align*}
and hence
\begin{align*}
\| \chi_k(h) \| \leq 1 \leq C \quad \text{for all $h \in H$.} 
\end{align*}
We define the open sets $U_k,V_k \subset H$ as 
\begin{align*}
U_k := \{ \| T_k \| > 1/4 \} \quad \text{and} \quad V_k := \{ \| T_k \| < 1/2 \}. 
\end{align*}
Then we have $H = U_k \cup V_k$ and $\chi_k(h) = 0$ for all $h \in V_k$. This shows $\chi_k \in C^{\infty}(H,\bbr)$, proving the first statement, and regarding the second statement, it suffices to show (\ref{est-chi-k}) for all $k \in \bbn$ and all $h \in U_k$. Let $k \in \bbn$ and all $h \in U_k$ be arbitrary. By (\ref{norm-est}) and (\ref{T-k-est}) we obtain
\begin{align*}
\| D(\eta \circ T_k)(h) \| &\leq \| D\eta(T_k h) \| \, \| DT_k h \| \leq \| T_k \| \leq L/r,
\\ \| D^2(\eta \circ T_k)(h) \| &\leq \| D^2 \eta(T_k h) \| \, \| D T_k h \|^2 + \| D\eta (T_k h) \|^2 \| D^2 T_k h \| \leq 8 L^2 / r^2,
\end{align*}
and hence
\begin{align*}
\| D \chi_k(h) \| &= \| D(\varphi \circ \eta \circ T_k)(h) \| \leq \| D \varphi (\eta(T_k h)) \| \, \| D (\eta \circ T_k)(h) \| \leq ML/r \leq C/r,
\\ \| D^2 \chi_k(h) \| &= \| D^2(\varphi \circ \eta \circ T_k)(h) \| \leq \| D^2 \varphi(\| T_k h \|) \| \, \| D(\eta \circ T_k)(h) \|^2
\\ &\quad + \| D \varphi(\| T_k h \|) \|^2 \| D^2 (\eta \circ T_k)(k) \| \leq M L^2 / r^2 + 8 M^2 L^2 / r^2 \leq C / r^2,
\end{align*}
completing the proof.
\end{proof}

The following auxiliary result extends Fact 7 in \cite{Fine-Approx}.

\begin{lemma}\label{lemma-C-b-2-part-2}
There exists a constant $M \in \bbr_+$ such that
\begin{align}\label{est-Psi-final}
\max \{ \| D \Psi(h) \|, r \| D^2 \Psi(h) \| \} \leq M \quad \text{for all $h \in H$.}
\end{align}
\end{lemma}

\begin{proof}
Let $C \in \bbr_+$ be the constant from Lemma \ref{lemma-chi-est}. We define the constant $M \in \bbr_+$ as
\begin{align*}
M := 3 \ubc C.
\end{align*}
Let $h \in H$ be arbitrary. Noting that $T_k h \to 0$ for $k \to \infty$, let $n \in \bbn$ be the smallest index such that 
\begin{align}\label{norm-T-1}
\| T_n h \| \leq 1. 
\end{align}
Then we have $\| T_k h \| > 1$ for all $k = 1,\ldots,n-1$. By the continuity of the linear operators $T_1,\ldots,T_{n-1}$, there exists $\delta > 0$ such that
\begin{align*}
\| T_k (h-g) \| > 1 \quad \text{for all $k = 1,\ldots,n-1$ and all $g \in H$ with $\| g \| \leq \delta$.}
\end{align*}
By the definition (\ref{def-chi}) of $\chi_k$ we obtain
\begin{align*}
\chi_k(h-g) = 1 \quad \text{for all $k = 1,\ldots,n-1$ and all $g \in H$ with $\| g \| \leq \delta$,}
\end{align*}
and it follows that
\begin{align}\label{chi-diff-zero}
D \chi_k(h) = 0 \quad \text{and} \quad D^2 \chi_k(h) = 0 \quad \text{for all $k = 1,\ldots,n-1$.}
\end{align}
Furthermore, by the definition (\ref{def-Psi-final}) of $\Psi$ we have
\begin{align*}
D \Psi(h) &= \sum_{k \in \bbn} D \chi_k(h) \langle e_k^*,h \rangle e_k + \sum_{k \in \bbn} \chi_k(h) \langle e_k^*, \cdot \rangle e_k,
\\ D^2 \Psi(h) &= \sum_{k \in \bbn} D^2 \chi_k(h) \langle e_k^*,h \rangle e_k + 2 \sum_{k \in \bbn} D \chi_k(h) \langle e_k^*,\cdot \rangle e_k,
\end{align*}
and hence, by (\ref{chi-diff-zero}), Lemmas \ref{lemma-norm-g-star}, \ref{lemma-chi-est} and (\ref{norm-T-1}) we obtain
\begin{align*}
\| D \Psi(h) \| &\leq \bigg\| \sum_{k \geq n} D \chi_k(h) \langle e_k^*,h \rangle e_k \bigg\| + \bigg\| \sum_{k \in \bbn} \chi_k(h) \langle e_k^*,\cdot \rangle e_k \bigg\|
\\ &\leq \ubc C/r \bigg\| \sum_{k \geq n} \langle e_k^*,h \rangle e_k \bigg\| + \ubc C \bigg\| \sum_{k \in \bbn} \langle e_k^*,\cdot \rangle e_k \bigg\|
\\ &\leq \ubc C \| T_n h \| + \ubc C \leq M,
\end{align*}
and similarly
\begin{align*}
\| D^2 \Psi(h) \| &\leq \bigg\| \sum_{k \geq n} D^2 \chi_k(h) \langle e_k^*,h \rangle e_k \bigg\| + 2 \bigg\| \sum_{h \in \bbn} D \chi_k(h) \langle e_k^*,\cdot \rangle e_k \bigg\|
\\ &\leq \ubc C / r^2 \bigg\| \sum_{k \geq n} \langle e_k^*,h \rangle e_k \bigg\| + 2 \ubc C / r \bigg\| \sum_{k \in \bbn} \langle e_k^*,\cdot \rangle e_k \bigg\|
\\ &\leq \ubc C/r \| T_n h \| + 2 \ubc C/r \leq M/r,
\end{align*}
completing the proof.
\end{proof}

\begin{lemma}\label{lemma-approx-C-b-2-b}
For all $\sigma \in \F(H) \cap C_b^{1,1}(H)$ and all sequences $a,b \in (0,\infty)^{\bbn}$, where $b$ is chosen such that (\ref{diff-hat-cap}) is fulfilled, and every constant $r > 0$ we have $\sigma^{(a,b,r)} \in C_b^2(H)$.
\end{lemma}

\begin{proof}
By Lemmas \ref{lemma-bar-sigma} and \ref{lemma-C-b-2-part-2} there exist constants $C,M \in \bbr_+$ such that we have (\ref{est-bar-sigma}) and (\ref{est-Psi-final}). Let $h \in H$ be arbitrary. By Lemma \ref{lemma-Psi-local} there exist $n \in \bbn$ and $\delta > 0$ such that we have (\ref{Psi-local}). Furthermore, by the definition (\ref{def-sigma-final}) of $\sigma^{(a,b,r)}$ and relation (\ref{restr-bar-sigma}) we have
\begin{align*}
\sigma^{(a,b,r)}(h-g) = \bar{\sigma}_n(\Psi(h-g)) \quad \text{for all $g \in H$ with $\| g \| \leq \delta$.}
\end{align*}
Therefore, and by estimates (\ref{est-bar-sigma}) and (\ref{est-Psi-final}), we obtain
\begin{align*}
\| \sigma^{(a,b,r)}(h) \| &= \| \bar{\sigma}_n(\Psi(h)) \| \leq C,
\\ \| D \sigma^{(a,b,r)}(h) \| &= \| D (\bar{\sigma}_n \circ \Psi)(h) \| \leq \| D \bar{\sigma}_n(\Psi(h)) \| \, \| D \Psi(h) \| \leq CM,
\\ \| D^2 \sigma^{(a,b,r)}(h) \| &= \| D^2 (\bar{\sigma}_n \circ \Psi)(h) \| \leq \| D^2 \bar{\sigma}_n (\Psi(h)) \| \, \| D \Psi(h) \|^2
\\ &\quad + \| D \bar{\sigma}_n (\Psi(h)) \|^2 \, \| D^2 \Psi(h) \| \leq CM^2 + C^2 M/r,
\end{align*}
finishing the proof.
\end{proof}

\begin{lemma}\label{lemma-approx-C-b-2-c}
Let $\sigma \in \F(H) \cap C_b^{1,1}(H)$ be a locally parallel function. Then, there exist a sequences $a^0,b^0 \in (0,\infty)^{\bbn}$, where $b^0$ is chosen such that (\ref{diff-hat-cap}) is fulfilled with $b$ replaced by $b^0$, such that for all sequences $a,b \in (0,\infty)^{\bbn}$ with $a^0 \leq_{\bbn} a$ and $b^0 \leq_{\bbn} b$ and every constant $r > 0$ the function $\sigma^{(a,b,r)} : H \to H$ is weakly locally parallel.
\end{lemma}

\begin{proof}
Since $\sigma$ is locally parallel, there exists $\epsilon > 0$ such that for all $(h^*,h) \in D$ we have (\ref{loc-par-prop}). Let $a^0 \in (0,\infty)^{\bbn}$ be the sequence given by $a_n^0 := 2 / \epsilon$ for each $n \in \bbn$. Furthermore, we choose $b^0 \in (0,\infty)^{\bbn}$ such that $b_n^0 \geq 4 / \epsilon$ for each $n \in \bbn$, and condition (\ref{diff-hat-cap}) is fulfilled with $b$ replaced by $b^0$. Let $a,b \in (0,\infty)^{\bbn}$ be arbitrary sequences with $a^0 \leq_{\bbn} a$ and $b^0 \leq_{\bbn} b$, and let $r > 0$ be an arbitrary constant. First, we will show that for all $n \in \bbn$ and all $(h^*,h) \in D$ we have
\begin{align}\label{cap-sigma-par}
\langle h^*,\Sigma_n(h-g) \rangle = 0 \quad \text{for all $g \in H$ with $\| g \| \leq \epsilon / 2$.}
\end{align}
For this purpose, let $g \in H$ with $\| g \| \leq \epsilon / 2$ be arbitrary. By the definition (\ref{def-cap-sigma}) of $\Sigma_n$, relation (\ref{loc-par-prop}), and since $\supp(\varphi) \subset [-1,1]$ and $a_n \geq 2 / \epsilon$, we obtain
\begin{align*}
\langle h^*,\Sigma_n(h-g) \rangle &= \frac{(a_n)^n}{c_n} \int_{E_n} \langle h^*, \sigma(h-g-f) \rangle \varphi(a_n \| f \|) df
\\ &= \frac{(a_n)^n}{c_n} \int_{E_n} \underbrace{\langle h^*, \sigma(h-(g+f)) \rangle}_{= 0} \varphi(a_n \| f \|) \bbI_{\{ \| f \| \leq \epsilon / 2 \}} df
\\ &\quad + \frac{(a_n)^n}{c_n} \int_{E_n} \langle g^*, \sigma(h-(g+f)) \rangle \underbrace{\varphi(a_n \| f \|)}_{= 0} \bbI_{\{ \| f \| > \epsilon / 2 \}} df = 0,
\end{align*}
showing (\ref{cap-sigma-par}). Noting the definition (\ref{def-hat-sigma}) of $\hat{\sigma}_n$, relation (\ref{cap-sigma-par}) and that $b_n \geq 4 / \epsilon$, analogously we show that for all $n \in \bbn$ and all $(h^*,h) \in D$ we have
\begin{align}\label{sigma-hat-par}
\langle h^*,\hat{\sigma}_n(h-g) \rangle = 0 \quad \text{for all $g \in H$ with $\| g \| \leq \epsilon / 4$.}
\end{align}
Next, we set $M := \bc \geq 1$. By induction, we will show that for all $n \in \bbn_0$ and all $(h^*,h) \in D$ we have
\begin{align}\label{sigma-bar-par}
\langle h^*,\bar{\sigma}_n(h-g) \rangle = 0 \quad \text{for all $g \in H$ with $\| g \| \leq \frac{\epsilon}{4 M^n}$.}
\end{align}
Relation (\ref{sigma-bar-par}) holds true for $n = 0$. Indeed, since $0 \in K$ and $0 \leq_K h$, by Assumption \ref{ass-D} we also have $(h^*,0) \in D$. Therefore, by the definition (\ref{def-sigma-0-bar}) of $\bar{\sigma}_0$, and since $\sigma$ is parallel, for all $g \in H$ with $\| g \| \leq \epsilon / 4$ we obtain
\begin{align*}
\langle h^*,\bar{\sigma}_0(h-g) \rangle = \langle h^*,\sigma(0) \rangle = 0.
\end{align*}
For the induction step, suppose that (\ref{sigma-bar-par}) is satisfied for $n-1$. Since $\Pi_{n-1} h \in K$ and $\Pi_{n-1} h \leq_K h$, by Assumption \ref{ass-D} we also have $(h^*,\Pi_{n-1} h) \in D$. Let $g \in H$ with $\| g \| \leq \frac{\epsilon}{4 M^n}$ be arbitrary. Then, we have
\begin{align*}
\| g \| \leq \frac{\epsilon}{4} \quad \text{and} \quad \| \Pi_{n-1} g \| \leq \frac{\epsilon}{4 M^{n-1}} \leq \frac{\epsilon}{4},
\end{align*}
and hence, by the definition (\ref{def-sigma-n-bar}) of $\bar{\sigma}_n$, relation (\ref{sigma-hat-par}) and the induction hypothesis, we obtain
\begin{align*}
&\langle h^*,\bar{\sigma}_n(h-g) \rangle
\\ &= \langle h^*,\hat{\sigma}_n(h-g) \rangle + \langle h^*,\bar{\sigma}_{n-1}(\Pi_{n-1} (h-g)) \rangle + \langle h^*,\hat{\sigma}_n(\Pi_{n-1} (h-g)) \rangle
\\ &= \langle h^*,\hat{\sigma}_n(h-g) \rangle + \langle h^*,\bar{\sigma}_{n-1}(\Pi_{n-1} h - \Pi_{n-1} g) \rangle + \langle h^*,\hat{\sigma}_n(\Pi_{n-1} h - \Pi_{n-1} g) \rangle = 0,
\end{align*}
proving (\ref{sigma-bar-par}). Now, let $(h^*,h) \in D$ be arbitrary. By the definition (\ref{def-Psi-final}) of $\Psi$ we have $\Psi(h) \in K$ and $\Psi(h) \leq_K h$, and hence, by Assumption \ref{ass-D} we also have $(h^*,\Psi(h)) \in D$. By Lemma \ref{lemma-Psi-local} there exist $n \in \bbn$ and $\delta > 0$ such that we have (\ref{Psi-local}), and there exists $C > 0$ such that
\begin{align*}
\| \Psi(h-g) - \Psi(h) \| \leq C \| g \| \quad \text{for all $g \in H$.}
\end{align*}
We define $\eta > 0$ as
\begin{align*}
\eta := \min \bigg\{ \delta, \frac{\epsilon}{4 M^n C} \bigg\}.
\end{align*}
Let $g \in H$ with $\| g \| \leq \eta$ be arbitrary. Then we have
\begin{align*}
\| \Psi(h-g) - \Psi(h) \| \leq \frac{\epsilon}{4 M^n},
\end{align*}
and hence, by the definition (\ref{def-sigma-final}) of $\sigma^{(a,b,r)}$, relation (\ref{restr-bar-sigma}) and (\ref{sigma-bar-par}) we obtain
\begin{align*}
\langle h^*,\sigma^{(a,b,r)}(h-g) \rangle &= \langle h^*,\bar{\sigma}(\Psi(h-g)) \rangle = \langle h^*,\bar{\sigma}_n(\Psi(h-g)) \rangle
\\ &= \langle h^*,\bar{\sigma}_n(\Psi(h) - (\Psi(h-g) - \Psi(h)) ) \rangle = 0,
\end{align*}
showing that $\sigma^{(a,b,r)}$ is weakly locally parallel.
\end{proof}

\begin{proposition}\label{prop-C-b-2}
Let $\sigma \in \F(H) \cap C_b^{1,1}(H)$ be a locally parallel function. Then there are a constant $L \in \bbr_+$ and a sequence
\begin{align*}
(\sigma_n)_{n \in \bbn} \subset \Lip_L(H) \cap \F(H) \cap C_b^2(H)
\end{align*}
such that $\sigma_n$ is weakly locally parallel for each $n \in \bbn$, and we have $\sigma_n \to \sigma$.
\end{proposition}

\begin{proof}
This is an immediate consequence of Lemmas \ref{lemma-approx-C-b-2-a}, \ref{lemma-C-b-2-pre}, \ref{lemma-approx-C-b-2-b}  and \ref{lemma-approx-C-b-2-c}.
\end{proof}

\end{appendix}


\begin{thebibliography}{20}

\bibitem{Approx} Azagra, D., Ferrera, J., L\'{o}pez-Mesas, F., Rangel, Y. (2007):
  Smooth approximation of Lipschitz functions on Riemannian manifolds.
  \textit{Journal of Mathematical Analysis and Applications} {\bf 124}(1), 47--66.

\bibitem{Fine-Approx} Azagra, D., Gil, J.~G., Jaramillo, J.~A., Lovo, M. (2005):
  $C^1$-fine approximation of functions on Banach spaces with unconditional basis.
  \textit{The Quarterly Journal of Mathematics} {\bf 56}(1), 13--20.

\bibitem{Bari} Bari, N.~K. (1951):
  Biorthogonal systems and bases in Hilbert Space.
  \textit{Uch. Zap. Mosk. Gos. Univ.} {\bf 148}, 69--107.

\bibitem{Bj_La} Bj{\"o}rk, T., Land{\'e}n, C. (2002):
  On the construction of finite dimensional realizations for nonlinear forward rate models.
  \textit{Finance and Stochastics} {\bf 6}(3), 303--331.

\bibitem{Bj_Sv} Bj{\"o}rk, T., Svensson, L. (2001):
  On the existence of finite dimensional realizations for nonlinear forward rate models.
  \textit{Mathematical Finance} {\bf 11}(2), 205--243.

\bibitem{Da_Prato} Da~Prato, G., Zabczyk, J. (1992):
  \textit{Stochastic equations in infinite dimensions.} Cambridge University Press, New York.

\bibitem{Fabian} Fabian, M., Habala, P., H\'{a}jek, P., Santaluc\'{i}a, V.~M., Pelant, J., Zizler, V. (2001):
  \textit{Functional analysis and infinite dimensional geometry.}
  Springer, New York.

\bibitem{Filipovic-inv} Filipovi\'c, D. (2000):
  Invariant manifolds for weak solutions to stochastic equations.
  \textit{Probability Theory and Related Fields} {\bf 118}(3), 323--341.

\bibitem{SPDE} Filipovi\'c, D., Tappe, S., Teichmann, J. (2010):
  Jump-diffusions in Hilbert spaces: Existence, stability and numerics.
  \textit{Stochastics} {\bf 82}(5), 475--520.

\bibitem{Positivity} Filipovi\'c, D., Tappe, S., Teichmann, J. (2010):
  Term structure models driven by Wiener processes and Poisson measures: Existence and positivity.
  \textit{SIAM Journal on Financial Mathematics} {\bf 1}(1), 523--554.

\bibitem{FTT-manifolds} Filipovi\'c, D., Tappe, S., Teichmann, J. (2014):
  Invariant manifolds with boundary for jump-diffusions.
  \textit{Electronic Journal of Probability} {\bf 19}(51), 1--28.

\bibitem{FTT-appendix} Filipovi\'c, D., Tappe, S., Teichmann, J.:
  Stochastic partial differential equations and submanifolds in Hilbert spaces. Appendix of \emph{Invariant manifolds with boundary for jump-diffusions}, (2014). {\tt (http://arxiv.org/abs/1202.1076v2)}

\bibitem{Filipovic} Filipovi\'c, D., Teichmann, J. (2003):
  Existence of invariant manifolds for stochastic equations in infinite dimension.
  \textit{Journal of Functional Analysis} {\bf 197}(2), 398--432.

\bibitem{Filipovic-Teichmann-royal} Filipovi\'c, D., Teichmann, J. (2004):
  On the geometry of the term structure of interest rates.
  \textit{Proceedings of The Royal Society of London. Series A. Mathematical, Physical and Engineering Sciences}
  {\bf 460}(2041), 129--167.

\bibitem{Fry} Fry, R. (2006):
  Approximation by $C^p$-smooth Lipschitz functions on Banach spaces. \textit{Journal of Mathematical Analysis and Applications} {\bf 315}(2), 599--605.

\bibitem{Atma-book} Gawarecki, L., Mandrekar, V. (2011):
  \textit{Stochastic differential equations in infinite dimensions with applications to SPDEs.} Springer, Berlin.

\bibitem{Hajek-Johanis-G} H\'{a}jek, P., Johanis, M. (2009):
  Uniformly G\^{a}teaux smooth approximation on $c_0(\Gamma)$.
  \textit{Journal of Mathematical Analysis and Applications} {\bf 350}(2), 623--629.

\bibitem{Hajek-Johanis} H\'{a}jek, P., Johanis, M. (2010):
  Smooth approximations.
  \textit{Journal of Functional Analysis} {\bf 259}(3), 561--582.

\bibitem{Jachimiak-note} Jachimiak, W. (1997):
  A note on invariance for semilinear differential equations.
  \textit{Bulletin of the Polish Academy of Sciences} {\bf 45}(2).

\bibitem{Jachimiak} Jachimiak, W. (1998):
  Stochastic invariance in infinite dimension.
  \textit{Polish Academy of Sciences}.

\bibitem{Jacod-Shiryaev}
  Jacod, J., Shiryaev, A.~N. (2003):
  \textit{Limit theorems for stochastic processes}.
  Springer, Berlin.

\bibitem{Johanis} Johanis, M. (2003):
  Approximation of Lipschitz mappings.
  \textit{Serdica Mathematical Journal} {\bf 29}(2), 141--148.

\bibitem{Lasry-Lions} Lasry, J.~M., Lions, P.~L. (1986):
  A remark on regularization in Hilbert spaces.
  \textit{Israel Journal of Mathematics} {\bf 55}(3), 257--266.

\bibitem{Liu-Roeckner} Liu, W., R\"{o}ckner, M. (2015): 
\textit{Stochastic partial differential equations: An introduction} Springer, Heidelberg.

\bibitem{MPR} Marinelli, C., Pr\'ev\^ot, C., R\"ockner, M. (2010): 
  Regular dependence on initial data for stochastic evolution equations with multiplicative Poisson noise.
  \textit{Journal of Functional Analysis} {\bf 258}(2), 616--649.

\bibitem{Milian} Milian, A. (2002):
  Comparison theorems for stochastic evolution equations.
  \textit{Stochastics and Stochastic Reports} {\bf 72}(1--2), 79--108.

\bibitem{Moulis} Moulis, N. (1971):
  Approximation de fonctions diff\'{e}rentiables sur certains espaces de Banach.
  \textit{Ann. Inst. Fourier (Grenoble)} {\bf 21}(4), 293--345.

\bibitem{Nakayama-Support} Nakayama, T. (2004):
  Support theorem for mild solutions of SDE's in Hilbert spaces.
  \textit{J. Math. Sci. Univ. Tokyo} {\bf 11}(3), 245--311.

\bibitem{Nakayama} Nakayama, T. (2004):
  Viability Theorem for SPDE's including HJM framework.
  \textit{J. Math. Sci. Univ. Tokyo} {\bf 11}(3), 313--324.

\bibitem{Pazy} Pazy, A. (1983):
  \textit{Semigroups of linear operators and applications to partial differential equations.} Springer, New York.

\bibitem{P-Z-book}
    Peszat, S., Zabczyk, J. (2007):
\textit{Stochastic partial differential equations with L\'evy noise}.
Cambridge University Press, Cambridge.

\bibitem{Platen-Tappe}
  Platen, E., Tappe, S. (2015):
  Real-world forward rate dynamics with affine realizations. \textit{Stochastic Analysis and Applications} {\bf 33}(4), 573--608.

\bibitem{Prevot-Roeckner} Pr\'{e}v\^{o}t, C., R\"{o}ckner, M. (2007): 
\textit{A concise course on stochastic partial differential equations.} Springer, Berlin.

\bibitem{Tappe-Wiener} Tappe, S. (2010):
  An alternative approach on the existence of affine realizations for HJM term structure models.
  \textit{Proceedings of The Royal Society of London. Series A. Mathematical, Physical and Engineering Sciences} {\bf 466}(2122), 3033--3060.

\bibitem{Tappe-Levy} Tappe, S. (2012):
  Existence of affine realizations for L\'{e}vy term structure models. 
  \textit{Proceedings of The Royal Society of London. Series A. Mathematical, Physical and Engineering Sciences} {\bf 468} (2147), 3685--3704.

\bibitem{Tappe-refine} Tappe, S. (2012): Some refinements of existence results for SPDEs driven by Wiener processes and Poisson random measures. \textit{International Journal of Stochastic Analysis}, vol. {\bf 2012}, Article ID 236327, 24~pages.

\bibitem{Tappe-affin-real} Tappe, S. (2015):
  Existence of affine realizations for stochastic partial differential equations driven by L\'{e}vy processes.
  \textit{Proceedings of The Royal Society of London. Series A. Mathematical, Physical and Engineering Sciences} {\bf 471}(2178).

\bibitem{Tappe-affine} Tappe, S. (2016): 
  Affine realizations with affine state processes for stochastic partial differential equations. \textit{Stochastic Processes and Their Applications} {\bf 126}(7), 2062--2091.

\end{thebibliography}
\end{document}